\documentclass[a4paper,10pt]{article}
\usepackage[utf8x]{inputenc}
\usepackage{tracefnt,amsmath,tabu,array}
\usepackage{amssymb,graphicx,setspace,amsfonts,amsbsy}
\usepackage{pifont,latexsym,ifthen,amsthm,rotating,calc,textcase,booktabs}

\newtheorem{theorem}{Theorem}[section]
\newtheorem{lemma}[theorem]{Lemma}
\newtheorem{corollary}[theorem]{Corollary}
\newtheorem{proposition}[theorem]{Proposition}
\newtheorem{definition}[theorem]{Definition}

\newtheorem{conjecture}[theorem]{Conjecture}
\newtheorem{remark}[theorem]{Remark}
\newcommand{\filledbox}{\leavevmode
  \hbox to.77778em{%
  \hfil\vbox to.675em{\hrule width.6em height.6em}\hfil}}

\newcommand{\Rm}{{\mathbb R}}

\newcommand{\eps}{\varepsilon}

\begin{document}
\tabulinesep=1.0mm
\title{Energy Distribution of Radial Solutions to Energy Subcritical Wave Equation with an Application on Scattering Theory\footnote{MSC classes: 35L71, 35L05; This work is supported by National Natural Science Foundation of China Programs 11601374, 11771325}}

\author{Ruipeng Shen\\
Centre for Applied Mathematics\\
Tianjin University\\
Tianjin, China}

\maketitle

\begin{abstract}
  The topic of this paper is a semi-linear, energy sub-critical, defocusing wave equation $\partial_t^2 u - \Delta u = - |u|^{p -1} u$ in the 3-dimensional space 
 ($3\leq p<5$) whose initial data are radial and come with a finite energy. We split the energy into inward and outward energies, then apply energy flux formula to obtain the following asymptotic distribution of energy: Unless the solution scatters, its energy can be divided into two parts: ``scattering energy'' which concentrates around the light cone $|x|=|t|$ and moves to infinity at the light speed and ``retarded energy'' which is at a distance of at least $|t|^\beta$ behind when $|t|$ is large. Here $\beta$ is an arbitrary constant smaller than $\beta_0(p) = \frac{2(p-2)}{p+1}$. A combination of this property with a more detailed version of the classic Morawetz estimate gives a scattering result under a weaker assumption on initial data $(u_0,u_1)$ than previously known results. More precisely, we assume  
 \[
  \int_{\Rm^3} (|x|^\kappa+1)\left(\frac{1}{2}|\nabla u_0|^2 + \frac{1}{2}|u_1|^2+\frac{1}{p+1}|u|^{p+1}\right) dx < +\infty.
 \]
Here $\kappa>\kappa_0(p) =1-\beta_0(p) = \frac{5-p}{p+1}$ is a constant. This condition is so weak that the initial data may be outside the critical Sobolev space of this equation. This phenomenon is not covered by previously known scattering theory, as far as the author knows.
\end{abstract}

\section{Introduction}

\subsection{Background}

In this work we consider the Cauchy problem of the defocusing semi-linear wave equation 
\[
 \left\{\begin{array}{ll} \partial_t^2 u - \Delta u = - |u|^{p-1}u, & (x,t) \in \Rm^3 \times \Rm; \\
 u(\cdot, 0) = u_0; & \\
 u_t (\cdot,0) = u_1. & \end{array}\right.\quad (CP1)
\]
If $u$ is a solution as above and $\lambda$ is a positive constant, then the function $u_\lambda = \lambda^{-2/(p-1)} u(x/\lambda,t/\lambda)$ is another solution to (CP1) with initial data 
\begin{align*}
& u_\lambda(\cdot,0) = \lambda^{-\frac{2}{p-1}} u_0(\cdot/\lambda);& &\partial_t u_\lambda(\cdot,0) = \lambda^{-\frac{2}{p-1}-1} u_1(\cdot/\lambda).&
\end{align*}
Both pairs of initial data share the same $\dot{H}^{s_p}\times \dot{H}^{s_p-1}(\Rm^3)$ norm if we choose $s_p = 3/2-2/(p-1)$. As a result, this Sobolev space is called the critical Sobolev space of this equation. It has been proved that this problem is locally well-posed for any initial data in this critical Sobolev space. Please see \cite{ls}, for instance, for more details. There is also an energy conservation law for suitable solutions:
\[
 E(u, u_t) = \int_{\Rm^3} \left(\frac{1}{2}|\nabla u(\cdot, t)|^2 +\frac{1}{2}|u_t(\cdot, t)|^2 + \frac{1}{p+1}|u(\cdot,t)|^{p+1}\right)\,dx = \hbox{Const}.
\]
The question about global behaviour of solutions is more difficult. In early 1990's M. Grillakis \cite{mg1} gave a satisfying answer in the energy critical case $p=5$: Any solution with initial data in the critical space $\dot{H}^1 \times L^2(\Rm^3)$ must scatter in both two time directions. In other words, the asymptotic behaviour of any solution mentioned above resembles that of a free wave. We expect that a similar result holds for other exponent $p$ as well.
\begin{conjecture}
 Any solution to (CP1) with initial data $(u_0,u_1) \in \dot{H}^{s_p} \times \dot{H}^{s_p-1}$ must exist for all time $t \in \Rm$ and scatter in both two time directions.
\end{conjecture}
\noindent This is still an open problem, although we do have progress in two different aspects:
\paragraph{Scattering result with a priori estimates} It has been proved that if a radial solution $u$ with a maximal lifespan $I$ satisfies an a priori estimate
 \begin{equation}
  \sup_{t \in I} \left\|(u(\cdot,t), u_t(\cdot, t))\right\|_{\dot{H}^{s_p} \times \dot{H}^{s_p-1} (\Rm^3)} < +\infty, \label{uniform UB}
 \end{equation}
 then $u$ is defined for all time $t$ and scatters. The proof uses a compactness-rigidity argument. The compactness part is nowadays a standard procedure in the study of wave and Sch\"{o}dinger equations; while the rigidity part does depend on specific situations. In fact, different methods were used for different range of $p$'s. The details can be found in Kenig-Merle \cite{km} for $p>5$, Shen \cite{shen2} for $3<p<5$ and Dodson-Lawrie \cite{cubic3dwave} for $1+\sqrt{2}<p\leq 3$. The author would also like to mention that the same result still holds in the non-radial case in the energy supercritical case $p>5$, as shown in the paper \cite{kv2}. Finally please pay attention that \eqref{uniform UB} is automatically true in the energy critical case $p=5$, as long as initial data are contained in the critical Sobolev space $\dot{H}^1 \times L^2$, thanks to the energy conservation law. 
 \paragraph{Strong Assumptions on Initial Data} There is also multiple scattering results if we assume that the initial data satisfy stronger regularity and/or decay conditions. These results are usually proved via a suitable global space-time integral estimate.
\begin{itemize}
 \item In the energy sub-critical case $3\leq p < 5$, the solutions always scatter if initial data satisfy an additional regularity-decay condition
  \begin{equation} \label{condition1}
    \int_{\Rm^3} \left[(|x|^2+1) (|\nabla u_0 (x)|^2 + |u_1(x)|^2) + |u_0(x)|^2 \right] dx < \infty.
  \end{equation}
The main tool is the following conformal conservation law
\[
  \frac{d}{dt} Q(t, u, u_t) = \frac{4(3-p)t}{p+1} \int_{\Rm^3} \left|u(x,t)\right|^{p+1} dx.
\]
Here $Q(t,\varphi,\psi) = Q_0(t, \varphi, \psi) + Q_1(t, \varphi)$ is called the conformal charge with
\begin{align*}
 Q_0(t,\varphi,\psi) &= \left\|x\psi + t \nabla \varphi \right\|_{L^2(\Rm^3)}^2 + \left\|(t\psi+2\varphi)\frac{x}{|x|} +|x|\nabla \varphi\right\|_{L^2(\Rm^3)}^2;\\
 Q_1(t, \varphi) &= \frac{2}{p+1}\int_{\Rm^3} (|x|^2+t^2)|\varphi(x,t)|^{p+1} dx.
\end{align*}
The assumption \eqref{condition1} guarantees the finiteness of conformal charge $Q(t,u,u_t)$ when $t=0$. It immediately gives a global space-time integral estimate
\[
\int_{|t|>1} \int_{\Rm^3} |u(x,t)|^{p+1}\,dxdt \lesssim_p \sup_{t\in \Rm} Q_1(t, u(\cdot,t)) \leq \sup_{t\in \Rm} Q(t,u,u_t) = Q(0,u_0,u_1)<+\infty,
\]
which then implies the scattering of solutions. For more details please see \cite{conformal2, conformal}. 
\item The author's previous work \cite{shen3} proved the scattering result for $3 \leq p<5$ if initial data $(u_0, u_1) \in \dot{H}^1 \times L^2$ are radial and satisfy
\begin{align*}
  \int_{\Rm^3} (|x|+1)^{1+2\eps}\left(|\nabla u_0|^2 + |u_1|^2\right) dx < \infty
\end{align*}
for an arbitrary constant $\eps>0$, by introducing a conformal transformation: If $u$ is a solution as assumed, then for any $t_0\in \Rm$, the function
\[
   v(y, \tau)  = \frac{\sinh |y|}{|y|} e^\tau u \left( e^\tau \frac{\sinh |y|}{|y|}\cdot y, t_0 + e^\tau \cosh |y|\right), \quad (y,\tau) \in \Rm^3 \times \Rm
 \]
solves another wave equation
\[
 v_{\tau \tau} - \Delta_y v = - \left(\frac{|y|}{\sinh |y|}\right)^{p-1} e^{-(p-3)\tau} |v|^{p-1}v.
\]
We then apply a Morawetz-type estimate on the solutions $v$ of the second equation and rewrite it in the form of original solutions $u$. This helps to give a global space-time integral $\|u\|_{L^{2(p-1)} L^{2(p-1)}(\Rm\times \Rm^3)} < +\infty$ and finishes the proof. 
\item In the author's recent work \cite{pushmorawetz} we proved the same scattering result for radial solutions under a weaker assumption on initial data
\[
 \int_{\Rm^3} (|x|^\kappa+1)\left(\frac{1}{2}|\nabla u_0|^2 + \frac{1}{2}|u_1|^2 + \frac{1}{p+1}|u_0|^{p+1} \right) < +\infty.
\]
Here $\kappa>\kappa_1(p) = \frac{3(5-p)}{p+3}$ is a constant. The proof  uses a detailed version of the classic Morawetz estimate (see Section \ref{sec:morawetz} below) to give a decay rate of the space-time integral 
\[
  \int_{-\infty}^{+\infty} \int_{|x|>R} \frac{|u|^{p+1}}{|x|} dx dt \lesssim R^{-\kappa}. 
\]
This then gives the same estimate $\|u\|_{L^{2(p-1)} L^{2(p-1)}(\Rm\times \Rm^3)} < +\infty$ and implies the scattering.
\end{itemize}

\subsection{Main Results}

In this paper we always consider radial solutions to (CP1) with a finite energy. Energy-subcriticality guarantees the global existence of the solutions.The goal of this work is two-fold. 
\begin{itemize}
 \item We want to understand the spatial distribution of the energy as $t$ goes to infinity. This gives plentiful information about the global behaviour of solutions. 
 \item If the energy of initial data satisfies an additional decay assumption, we prove the scattering results.
\end{itemize}
In this subsection we give two main theorems. 

\begin{theorem} \label{main 1}
Assume $3\leq p<5$. Let $u$ be a radial solution to (CP1) with a finite energy $E$. Then there exist a three-dimensional free wave $v^-(x,t)$, with an energy $\tilde{E}_-\leq E$, so that 
 \begin{itemize}
  \item We have scattering outside any backward light cone ($R \in \Rm$)
   \[
     \lim_{t \rightarrow - \infty} \left\|\left(\nabla u(\cdot,t), u_t(\cdot,t)\right)- \left(\nabla v^- (\cdot,t), v_t^- (\cdot,t)\right)\right\|_{L^2(\{x\in \Rm^3:|x|>R+|t|\})} = 0.
   \]
   \item If we have $\tilde{E}_- = E$, then the scattering happens in the whole space in the negative time direction.
    \[
     \lim_{t \rightarrow - \infty} \left\|\left(u(\cdot,t), u_t(\cdot,t)\right)- \left(v^- (\cdot,t), v_t^- (\cdot,t)\right)\right\|_{\dot{H}^1\times L^2(\Rm^3)} = 0.
   \]
   \item If $\tilde{E}_- < E$, the remaining energy (also called ``retarded energy'') can be located: for any constants $c \in (0,1)$ and $\beta<\frac{2(p-2)}{p+1}$ we have
   \[
    \lim_{t \rightarrow - \infty} \int_{c|t|<|x|<|t|-|t|^\beta} \left(\frac{1}{2}|\nabla u(x,t)|^2 + \frac{1}{2}|u_t(x,t)|^2 + \frac{1}{p+1}|u(x,t)|^{p+1}\right) dx = E - \tilde{E}_-.
   \]
 \end{itemize}
 The asymptotic behaviour in the positive time direction is similar.
\end{theorem}
\begin{remark}
 If $\tilde{E}_-<E$, then the energy distribution is illustrated in figure \ref{energydist}. The ``gap'' between scattering energy, which travels at the light speed, and ``retarded energy'', which travels slightly slower, becomes wider and wider as time goes to infinity.
\end{remark}

 \begin{figure}[h]
 \centering
 \includegraphics[scale=0.9]{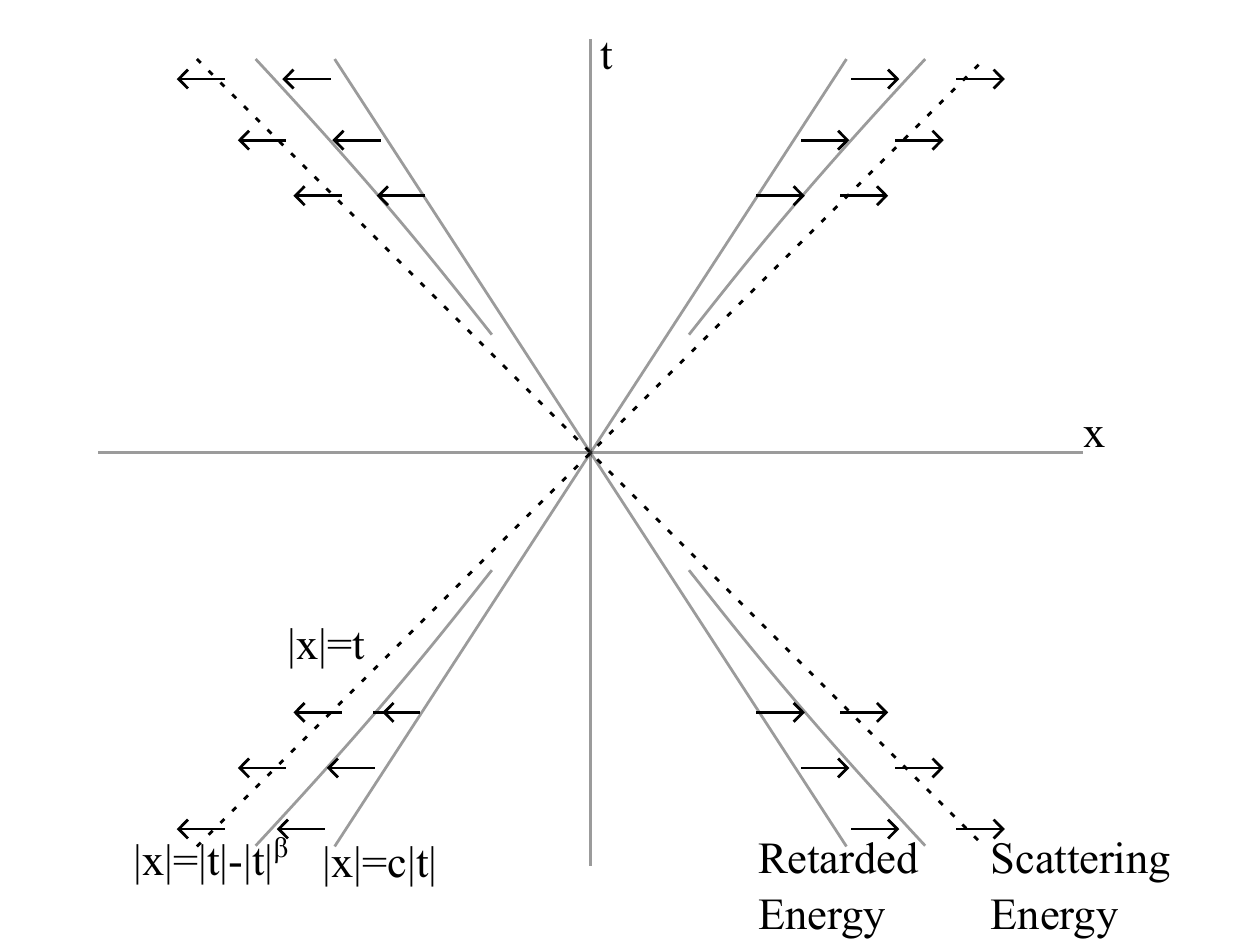}
 \caption{Illustration of travelling energy} \label{energydist}
\end{figure}

\begin{theorem} \label{main 2}
Assume $3\leq p<5$. Let $\kappa>\kappa_0(p)=\frac{5-p}{p+1}$ be a constant. If $u$ is a radial solution to (CP1) with initial data $(u_0,u_1)$ so that
 \[
  \int_{\Rm^3} (|x|^\kappa+1)\left(\frac{1}{2}|\nabla u_0|^2 + \frac{1}{2}|u_1|^2+\frac{1}{p+1}|u_0|^{p+1}\right) dx < +\infty,
 \]
 then $u$ must scatter in both two time directions. More precisely, there exists $(v_0^\pm ,v_1^\pm) \in \dot{H}^1 \times L^2(\Rm^3)$, so that 
 \[
  \lim_{t \rightarrow \pm \infty} \left\|\begin{pmatrix} u(\cdot,t)\\ \partial_t u(\cdot,t)\end{pmatrix} - 
  \mathbf{S}_L (t)\begin{pmatrix}u_0^\pm \\ u_1^\pm\end{pmatrix}\right\|_{\dot{H}^1 \times L^2(\Rm^3)} = 0.
 \]
 Here $\mathbf{S}_L (t)$ is the linear wave propagation operator.
\end{theorem}
\begin{remark}
The assumptions in our main theorems can not guarantee that $(u_0,u_1) \in \dot{H}^{s_p}\times \dot{H}^{s_p-1}$. For example, we can choose a radial function $u_0 \in C^\infty(\Rm^3)$ with decay
\begin{align*}
 &u_0(x) \simeq |x|^{-\frac{2(p+4)}{(p+1)^2}-\eps}; & &|\nabla u_0(x)| \simeq |x|^{-\frac{2(p+4)}{(p+1)^2}-1-\eps};& &|x|\gg 1.&
\end{align*}
Here $\eps$ is an sufficiently small positive constant. One can check that $(u_0,0)$ satisfies all the assumptions on initial data in two main theorems but $u_0 \notin L^{3(p-1)/2}(\Rm^3)$.  The latter implies that $u_0 \notin \dot{H}^{s_p}(\Rm^3)$ since we have the Sobolev embedding $\dot{H}^{s_p}(\Rm^3) \hookrightarrow L^{3(p-1)/2}(\Rm^3)$. As a result we have $(u(\cdot,t),u_t(\cdot,t)) \notin \dot{H}^{s_p} \times \dot{H}^{s_p-1}(\Rm^3)$ for any time $t$. This is the reason why in Theorem \ref{main 2} we measure the distance of $u$ and free waves by $\dot{H}^1 \times L^2$ norm instead. This is a phenomenon which has not been covered by previous results mentioned above. All the solutions discussed in those results come with initial data $(u_0,u_1) \in \dot{H}^{s_p} \times \dot{H}^{s_p-1}(\Rm^3)$.
\end{remark}

\subsection{The idea}

In this subsection we give the main idea of this paper and outline the proof of main theorems. The details can be found in later sections. 
 
\paragraph{Transformation to 1D} In order to take full advantage of our radial assumption, we use the following transformation: if $u$ is a radial solution to (CP1), then $w(r,t) = r u(x,t)$, where $|x|=r$, is a solution to one-dimensional wave equation
\[
 w_{tt} - w_{rr} = - \frac{|w|^{p-1}w}{r^{p-1}}.
\]
A basic calculation shows that 
\begin{align}
 &2\pi \int_a^b (|w_r(r,t)|^2+|w_t(r,t)|^2) dr \nonumber \\
 &\qquad = 2\pi\left[\int_a^b \left(r^2|u_r(r,t)|^2 + r^2|u_t(r,t)|^2\right)dr +b|u(b,t)|^2 - a|u(a,t)|^2\right]. \label{energy transformation}
\end{align}
Since for any radial $\dot{H}^1(\Rm^3)$ function $f(r)$, we have 
\begin{equation*}
 \lim_{r\rightarrow 0^+} r|f(r)|^2  = \lim_{r\rightarrow +\infty} r|f(r)|^2 = 0.
\end{equation*}
It immediately follows that
\begin{equation}
 2\pi \int_0^\infty (|w_r(r,t)|^2+|w_t(r,t)|^2) dr = \int_{\Rm^3} \left(\frac{1}{2}|\nabla u|^2 + \frac{1}{2} |u_t|^2 \right) dx. \label{energy between u and w}
\end{equation}
The new solution $w$ also satisfies an energy conservation law
\begin{align*}
 E(w,w_t) \doteq 2\pi \int_0^\infty (|w_r(r,t)|^2+|w_t(r,t)|^2+\frac{2}{p+1}\cdot \frac{|w(r,t)|^{p+1}}{r^{r-1}}) dr = E(u,u_t) = \hbox{Const}.
\end{align*}
The main tool to understand spatial energy distribution is the energy flux formula for inward and outward energy. 

\paragraph{Inward and Outward Energy} Let us first define

\begin{definition} \label{energies}
Let $w$ be a solution as above. We define the inward and outward energy
\begin{align*}
E_- (t) & = \pi \int_{0}^{\infty} \left(\left|w_r(r,t)+w_t(r,t)\right|^2+\frac{2}{p+1}\cdot\frac{|w(r,t)|^{p+1}}{r^{p-1}}\right) dr\\
 E_+ (t) & = \pi \int_{0}^{\infty} \left(\left|w_r(r,t)-w_t(r,t)\right|^2+\frac{2}{p+1}\cdot\frac{|w(r,t)|^{p+1}}{r^{p-1}}\right) dr
\end{align*}
We also need to consider their truncated versions
\begin{align*}
 E_- (t;r_1,r_2) & = \pi \int_{r_1}^{r_2} \left(\left|w_r(r,t)+w_t(r,t)\right|^2+\frac{2}{p+1}\cdot\frac{|w(r,t)|^{p+1}}{r^{p-1}}\right) dr\\
 E_+ (t;r_1,r_2) & = \pi \int_{r_1}^{r_2} \left(\left|w_r(r,t)-w_t(r,t)\right|^2+\frac{2}{p+1}\cdot\frac{|w(r,t)|^{p+1}}{r^{p-1}}\right) dr
 \end{align*}
\end{definition}
\begin{remark}
 The inward energy travels toward the origin as time $t$ increases; the outward energy travels in the opposite direction, as indicated by their names.  
\end{remark}

\paragraph{Energy Flux} We consider flux of inward and outward energies through either characteristic lines $t+r = \hbox{Const}$, $t-r = \hbox{Const}$ or the $t$-axis. This helps to give the following results 

\begin{itemize}
 \item[(a)] Almost all energy is outward energy as $t\rightarrow +\infty$.
 \begin{align*}
  &\lim_{t\rightarrow +\infty} E_-(t) = 0;& &\lim_{t\rightarrow +\infty} E_+(t) = E.&
 \end{align*}
 \item[(b)] Almost all energy is inward energy as $t\rightarrow -\infty$.
 \begin{align*}
  &\lim_{t\rightarrow -\infty} E_+(t) = 0;& &\lim_{t\rightarrow +\infty} E_-(t) = E.&
 \end{align*}
  \item[(c)] The energy flux through characteristic lines are bounded from the above. In particular, the following inequalities hold for any $s, \tau \in \Rm$.
 \begin{align*}
  &\int_{-\infty}^s \frac{|w(s-t,t)|^{p+1}}{(s-t)^{p-1}} dt \lesssim_p E;& &\int_\tau^\infty \frac{|w(t-\tau,t)|^{p+1}}{(t-\tau)^{p-1}} dt \lesssim_p E.&
 \end{align*}
 \end{itemize}
\paragraph{Asymptotic Behaviour} The characteristic line method gives
\begin{align*}
 &\frac{\partial}{\partial t}\left[(w_r+w_t)(s-t,t)\right] = -\frac{|w|^{p-1}w(s-t,t)}{(s-t)^{p-1}};\\
 &\frac{\partial}{\partial t}\left[(w_r-w_t)(t-\tau,t)\right] = \frac{|w|^{p-1}w(t-\tau,t)}{(t-\tau)^{p-1}}.
\end{align*}
Combining these two identities with part (c) above, we have the following convergence holds uniformly in any bounded interval. 
\begin{align*}
  &\lim_{t \rightarrow -\infty} (w_r+w_t)(s-t,t) = g_-(s);& &\lim_{t \rightarrow +\infty} (w_r-w_t)(t-\tau,t) = g_+(\tau).&
\end{align*}
We also have the following $L^2$ convergence for any $s_0,\tau_0\in \Rm$.
\begin{align*}
 & \left\|w_r(t-\tau,t) - \frac{1}{2}g_+(\tau)\right\|_{L^2((-\infty,\tau_0])} \rightarrow 0,& &\left\|w_t(t-\tau,t) +\frac{1}{2}g_+(\tau)\right\|_{L^2((-\infty,\tau_0])}\rightarrow 0,& &\hbox{as} \; t\rightarrow +\infty;&\\
 & \left\|w_r(s-t,t) - \frac{1}{2}g_-(s)\right\|_{L^2([s_0,\infty))} \rightarrow 0,& &\left\|w_t(s-t,t) -\frac{1}{2}g_-(s)\right\|_{L^2([s_0,\infty))}\rightarrow 0,& &\hbox{as} \; t\rightarrow -\infty.&
\end{align*}
This gives us the free waves $v^-(x,t)$ and $v^+(x,t)$ in the conclusion part (a) of theorem \ref{main 1}.  
\begin{align*}
 &v^-(x,t) = \frac{1}{2|x|}\int_{t-|x|}^{t+|x|} g_-(s) ds;& &v^+(x,t) = \frac{1}{2|x|}\int_{t-|x|}^{t+|x|} g_+(\tau) d\tau.&
\end{align*}
We can then prove Part (b) and (c) by considering the energy located in different regions via the energy flux formula. More details are given in Section \ref{sec:distribution}.

\paragraph{Morawetz Estimates} Another important ingredient of Theorem \ref{main 2}'s proof is a more detailed version of the classic Morawetz estimate as given below, which plays an key role in the author's recent work \cite{pushmorawetz} as well. This is a little different from the original inequality given by Perthame and Vega in the work \cite{benoit}, but can be deduced from the original one without difficulty. Please see Subsection \ref{sec:morawetz} for more details.
\begin{align*}
 & \frac{1}{2R}\int_{-\infty}^{+\infty} \int_{|x|<R}\left(\frac{1}{2}|\nabla u|^2+\frac{1}{2}|u_t|^2+\frac{1}{p+1}|u|^{p+1}\right) dx dt + \frac{1}{4R^2} \int_{-\infty}^{+\infty} \int_{|x|=R} |u|^2 d\sigma_R dt \\
 & \qquad + \frac{p-3}{2(p+1)R} \int_{-\infty}^{+\infty} \int_{|x|<R} |u|^{p+1} dx dt +\frac{p-1}{2(p+1)} \int_{-\infty}^{+\infty} \int_{|x|>R} \frac{|u|^{p+1}}{|x|} dx dt  \leq E. 
\end{align*}
If we pick up the last term in the left hand side and make $R\rightarrow 0^+$, we obtain the most frequently used Morawetz estimate:
\[
 \int_{-\infty}^{+\infty} \int_{\Rm^3} \frac{|u|^{p+1}}{|x|} dx dt  \leq CE.
\]
In this work, however, we choose large radius $R$ in the long inequality above and observe an important fact:  The first term in the left hand side itself is almost equal to $E$. This is because for almost all time $t \in [-R,R]$, as long as $t$ is not too closed to $R$ or $-R$, almost all energy concentrates in the ball of radius $R$, thanks to finite speed of propagation. As a result, we discard all other terms in the left hand and focus on the first term:
\[ 
 \int_{-\infty}^{+\infty} \int_{|x|<R}\left(\frac{1}{2}|\nabla u|^2+\frac{1}{2}|u_t|^2+\frac{1}{p+1}|u|^{p+1}\right) dx dt \leq 2RE.
\]
We can substitute the right hand side by $\int_{-R}^R \int_{\Rm^3} \left(\frac{1}{2}|\nabla u|^2+\frac{1}{2}|u_t|^2+\frac{1}{p+1}|u|^{p+1}\right) dx dt$, and rewrite the inequality in another form
\begin{align*}
 &\int_{|t|>R} \int_{|x|<R}\left(\frac{1}{2}|\nabla u|^2+\frac{1}{2}|u_t|^2+\frac{1}{p+1}|u|^{p+1}\right) dx dt\\
  & \qquad \leq \int_{-R}^{+R} \int_{|x|>R}\left(\frac{1}{2}|\nabla u|^2+\frac{1}{2}|u_t|^2+\frac{1}{p+1}|u|^{p+1}\right) dx dt.
\end{align*}
In physics this means that the total contribution of ``retarded energy'' when $t>|R|$ is always smaller or equal to the total contribution of energy which escapes the ball of radius $R$ in the time interval $[-R,R]$. This enables us to prove Theorem \ref{main 2} by a contradiction. On one hand, our theory on energy distribution of solutions gives a lower bound of the contribution of ``retarded energy'', unless the solution scatters. On the other hand, if we assume that the initial data satisfy a suitable decay condition, we can find an upper bound of the escaping energy. Thus we can simply make $R\rightarrow +\infty$ and compare the upper and lower bounds to finish the proof.
 
 \subsection{The Structure of This Paper}
 
This paper is organized as follows. In section 2 we collect notations, recall the classic Morawetz estimates and give a few preliminary results. Next in Section 3 we give a general formula for inward and outward energy flux. This helps to prove the energy distribution properties of the solutions in Section 4. Finally we prove the scattering of the solution $u$ under an additional decay assumption in the last section.
 
\section{Preliminary Results}

\subsection{Notations}

\paragraph{The $\lesssim$ symbol} We use the notation $A \lesssim B$ if there exists a constant $c$, so that the inequality $A \leq c B$ always holds.  In addition, a subscript of the symbol $\lesssim$ indicates that the constant $c$ is determined by the parameter(s) mentioned in the subscript but nothing else. In particular, $\lesssim_1$ means that the constant $c$ is an absolute constant. 

\paragraph{Radial functions} Let $u(x,t)$ be a spatially radial function. By convention $u(r,t)$ represents the value of $u(x,t)$ when $|x| = r$. 


\subsection{Uniform Pointwise Estimates}
In this subsection we first recall
\begin{lemma}[Please see Lemma 3.2 of Kenig and Merle's work \cite{km}] 
If $u$ is a radial $\dot{H}^1(\Rm^3)$ function, then 
\[
 |u(r)| \lesssim_1 r^{-1/2} \|u\|_{\dot{H}^1(\Rm^3)}.
\]
\end{lemma}
\noindent Therefore we always have $|w(r,t)| \lesssim_1 E^{1/2} r^{1/2}$. Because $u$ is not only a $\dot{H}^1(\Rm^3)$ function but also an $L^{p+1}(\Rm^3)$ function, this can be further improved if $r$ is large.
\begin{lemma} \label{best upper bound pointwise}
 If $w: [0,\infty) \rightarrow \Rm$ satisfies 
 \[
  2\pi \int_{0}^\infty \left(|w_r(r)|^2 + \frac{2}{p+1}\frac{|w(r)|^{p+1}}{r^{p-1}}\right)dr \leq E,
 \]
 then we have $|w(r)| \lesssim_{p} E^{2/(p+3)} r^{(p-1)/(p+3)} $.
\end{lemma}
\begin{proof}
First of all, we observe 
 \[
  |w(r)-w(r_0)| = \left|\int_{r_0}^r w_r(s) ds \right| \leq |r-r_0|^{1/2} \left(\int_{r_0}^r |w_r(s)|^2 ds\right)^{1/2} \leq (E/2\pi)^{1/2} |r-r_0|^{1/2}.
 \]
 Thus $w(r)$ converges as $r\rightarrow 0^+$. Since the singular integral of $|w(r)|^{p+1}/r^{p-1}$ near zero converges, it is clear that $w(0)=0$. Plugging $r_0=0$ in the estimate above, we re-discover the pointwise estimate $|w(r)|<(Er)^{1/2}$. Assume $|w(r_0)| = S \leq (Er_0)^{1/2}$.  By the inequality above we have $|w(r)-w(r_0)|<S/2$ thus $|w(r)|>S/2$ for $r \in [r_0,r_0+S^2/E]$.  As a result
\begin{align*}
E \geq \int_{r_0}^{r_0+S^2/E} \frac{|w(r)|^{p+1}}{r^{p-1}} \geq \frac{(S/2)^{p+1}}{(r_0+S^2/E)^{p-1}}\cdot S^2/E \geq \frac{(S/2)^{p+1}}{(2r_0)^{p-1}}\cdot S^2/E.
\end{align*} 
This means $E \gtrsim_p S^{p+3}/(r_0^{p-1}E)$. Thus we have $S \lesssim_p E^{2/(p+3)} r_0^{(p-1)/(p+3)}$ and finish the proof.
\end{proof}


\subsection{Morawetz Estimates} \label{sec:morawetz}
 
\begin{theorem} [Please see Perthame and Vega's work \cite{benoit}, we use the 3-dimensional case]
Let $u$ be a solution to (CP1) defined in a time interval $[0,T]$ with a finite energy $E$. Then we have the following inequality for any $R>0$
\begin{align}
 & \frac{1}{2R}\int_0^T \!\!\int_{|x|<R}(|\nabla u|^2+|u_t|^2) dx dt + \frac{1}{2R^2} \int_0^T \!\!\int_{|x|=R} |u|^2 d\sigma_R dt + \frac{p-2}{(p+1)R} \int_0^T \!\!\int_{|x|<R} |u|^{p+1} dx dt \nonumber \\
 & \qquad + \frac{p-1}{p+1} \int_0^T \int_{|x|>R} \frac{|u|^{p+1}}{|x|} dx dt + \frac{1}{R^2} \int_{|x|<R} |u(x,T)|^2 dx \leq 2E. \label{morawetz}
\end{align}
\end{theorem}
\begin{remark}
The notations $p$ and $E$ represent slightly different constants in the original paper \cite{benoit} and this current paper. Here we rewrite the inequality in the setting of the current work. The coefficient before the integral $\int_{B(0,R)} |u(T)|^2 dx$ was $\frac{d^2-1}{4R^2}$ (in the 3-dimensional case $\frac{2}{R^2}$) in the original paper. But the author believes that this is a minor typing mistake. It should have been $\frac{d^2-1}{8R^2}$ instead. 
\end{remark} 
\begin{remark}
 The upper bound of time interval $T$ does not appear in any of the coefficients above. We also have an energy conservation law. As a result, we can substitute the time interval $[0,T]$ by any bounded time interval $[T_1,T_2]$ or even $(-\infty,T]$. If we ignore the final term on the left hand side, we can all use the time interval $(-\infty,\infty)$.
 \begin{align}
 & \frac{1}{2R}\int_{-\infty}^{\infty} \!\int_{|x|<R}(|\nabla u|^2+|u_t|^2) dx dt + \frac{1}{2R^2} \int_{-\infty}^{\infty} \!\int_{|x|=R} |u|^2 d\sigma_R dt\nonumber \\
  & \qquad + \frac{p-2}{(p+1)R} \int_{-\infty}^{\infty} \!\int_{|x|<R} |u|^{p+1} dx dt + \frac{p-1}{p+1} \int_{-\infty}^{\infty} \int_{|x|>R} \frac{|u|^{p+1}}{|x|} dx dt  \leq 2E. \label{morawetz1}
\end{align}
\end{remark}
\noindent This Morawetz estimate plays two different roles in this work. On one hand, it gives a few global integral estimates, which our theory of energy distribution depends on. One of the most popular ones is $\int_{-\infty}^\infty \int_{\Rm^3} \frac{|u(x,t)|^{p+1}}{|x|} dx dt \lesssim E$. On the other hand, this Morawetz estimate also gives direct information on energy distribution, as we mentioned in the introduction part. Let us discuss these aspects one-by-one.
\paragraph{Global Estimates} 
If we pick the second term in the left hand side of \eqref{morawetz1} and use the radial assumption, we obtain
\[
 \sup_{r>0} \int_{-\infty}^{+\infty} |u(r,t)|^2 dt \leq \frac{E}{\pi}.
\]
We recall $w = ru$, use the Morawetz estimate \eqref{morawetz1} again and obtain
\begin{align*}
 & \int_{-\infty}^{+\infty} \int_0^R (|w_r(r,t)|^2+|w_t(r,t)|^2) dr dt \\
 & \qquad = \int_{-\infty}^{+\infty} \int_0^R r^2 (|u_r(r,t)|^2 +|u_t(r,t)|^2) dr dt + \int_{-\infty}^{+\infty} R|u(R,t)|^2 dt \leq \frac{RE}{\pi}
\end{align*}
We can also pick the third term of \eqref{morawetz1} and rewrite it in term of $w$
\[
 \int_{-\infty}^{+\infty} \int_0^R \frac{|w|^{p+1}}{r^{p-1}} dr dt \leq \frac{(p+1)}{2(p-2)\pi}RE.
\]
Finally we pick the forth term, make $R\rightarrow 0^+$ and rewrite it in term of $w$
\[
 \int_{-\infty}^{+\infty} \int_{\Rm^3} \frac{|u|^{p+1}}{|x|} dx dt \leq \frac{2(p+1)}{p-1}E 
 \Rightarrow \int_{-\infty}^{+\infty} \int_0^\infty \frac{|w|^{p+1}}{r^p} dr dt \leq \frac{p+1}{2(p-1)\pi} E.
\]
In summary, we have (The second line immediately follows the first line)
\begin{corollary} \label{cor w morawetz}
 Let $u$ be a radial solution to (CP1) with a finite energy $E$. Then $u$ and $w = ru$ satisfy 
 \begin{align*}
   \int_{-\infty}^{+\infty} \int_0^R \left(|w_r(r,t)|^2+|w_t(r,t)|^2+\frac{|w(r,t)|^{p+1}}{r^{p-1}}\right) dr dt & \lesssim_{p} RE;\\
   \liminf_{r\rightarrow 0^+} \int_{-\infty}^{+\infty} \left(|w_r(r,t)|^2+|w_t(r,t)|^2+\frac{|w(r,t)|^{p+1}}{r^{p-1}}\right) dt & \lesssim_p E;\\
   \int_{-\infty}^{+\infty} \int_0^\infty \frac{|w(r,t)|^{p+1}}{r^p} dr dt & \lesssim_{p} E.\\
   \sup_{r>0} \int_{-\infty}^{+\infty} |u(r,t)|^2 dt & \lesssim_p E.
 \end{align*}
\end{corollary}

\paragraph{Energy Distribution Information} We can combine part of the third term of \eqref{morawetz1} with the first term, divide both sides by $2$ and obtain 
\begin{align}
 & \frac{1}{2R}\int_{-\infty}^{+\infty} \int_{|x|<R}\left(\frac{1}{2}|\nabla u|^2+\frac{1}{2}|u_t|^2+\frac{1}{p+1}|u|^{p+1}\right) dx dt + \frac{1}{4R^2} \int_{-\infty}^{+\infty} \int_{|x|=R} |u|^2 d\sigma_R dt\nonumber\\
 & \qquad + \frac{p-3}{2(p+1)R} \int_{-\infty}^{+\infty} \int_{|x|<R} |u|^{p+1} dx dt +\frac{p-1}{2(p+1)} \int_{-\infty}^{+\infty} \int_{|x|>R} \frac{|u|^{p+1}}{|x|} dx dt  \leq E. \label{new morawetz}
\end{align}
This helps to give a decay rate $\int_{-\infty}^\infty \int_{|x|>R} \frac{|u(x,t)|^{p+1}}{|x|} dx dt \lesssim R^{-\kappa}$ in the author's recent work \cite{pushmorawetz}. In this work, we only need a weaker version of this inequality. Since every term in the left hand side is nonnegative, we can focus on the first term and obtain
 \[
   \int_{-\infty}^{+\infty} \int_{|x|<R}\left(\frac{1}{2}|\nabla u|^2+\frac{1}{2}|u_t|^2+\frac{1}{p+1}|u|^{p+1}\right) dx dt \leq 2RE.
 \]
 We substitute the right hand side by $\int_{-R}^R \int_{\Rm^3} \left(\frac{1}{2}|\nabla u|^2+\frac{1}{2}|u_t|^2+\frac{1}{p+1}|u|^{p+1}\right) dx dt$, split the integral in the left hand side into two parts, with time $t\in [-R,R]$ and $|t|>R$ respectively, combine the first part with the right hand side and finally obtain 
\begin{proposition} \label{energy distribution by morawetz}
 For any $R>0$, we have 
 \begin{align*}
 &\int_{|t|>R} \int_{|x|<R}\left(\frac{1}{2}|\nabla u|^2+\frac{1}{2}|u_t|^2+\frac{1}{p+1}|u|^{p+1}\right) dx dt\\
  & \qquad \leq \int_{-R}^{+R} \int_{|x|>R}\left(\frac{1}{2}|\nabla u|^2+\frac{1}{2}|u_t|^2+\frac{1}{p+1}|u|^{p+1}\right) dx dt.
\end{align*}
\end{proposition}
\noindent The following will not be used in the argument of this paper, instead it is a corollary of our theory on energy distribution. 
\begin{remark}
 A careful review of Perthame and Vega's calculation shows that for a radial solution $u$, we actually have an identity for any $R>0$ and $T_1<T_2$
 \begin{align*}
  & \frac{1}{2R}\int_{T_1}^{T_2} \int_{|x|<R}\left(\frac{1}{2}|\nabla u|^2+\frac{1}{2}|u_t|^2+\frac{1}{p+1}|u|^{p+1}\right) dx dt + \frac{1}{4R^2} \int_{T_1}^{T_2} \int_{|x|=R} |u|^2 d\sigma_R dt\nonumber\\
 & \qquad + \frac{p-3}{2(p+1)R} \int_{T_1}^{T_2} \int_{|x|<R} |u|^{p+1} dx dt +\frac{p-1}{2(p+1)} \int_{T_1}^{T_2} \int_{|x|>R} \frac{|u|^{p+1}}{|x|} dx dt  \\
 = &  2\pi \left[\int_0^R \frac{r}{R} w_t(r,T_1) w_r(r,T_1) dr + \int_R^\infty w_t(r,T_1)w_r(r,T_1) dr \right]\\
 & \qquad -2\pi \left[\int_0^R \frac{r}{R} w_t(r,T_2) w_r(r,T_2) dr + \int_R^\infty w_t(r,T_2)w_r(r,T_2) dr \right]
 \end{align*}
 By the identity $w_t w_r = \frac{1}{4}\left[(w_r+w_t)^2 - (w_r-w_t)^2\right]$, Proposition \ref{monotonicity of energies}, Corollary \ref{asymptotic behaviour 2} and Lemma \ref{lower bound of speed}, we know that if we make $T_1\rightarrow -\infty$, $T_2 \rightarrow +\infty$, then the limit of the right hand side is exactly $E$. Therefore the inequality \eqref{new morawetz} is actually an identity. Please note that the radial assumption is essential because we discard a term in the form of $\int_{T_1}^{T_2} \int_{|x|>R} {\mathbf D} u \cdot {\mathbf D}^2 \Phi \cdot {\mathbf D} u \,dx dt$ in the left hand side, where ${\mathbf D}^2 \Phi(x)$ is a positive semidefinite matrix whose eigenvalue $0$ has a single eigenvector $x$. This term vanishes only for radial solutions. 
\end{remark}
\section{Energy Flux for Inward and Outward Energies}

In this section we consider the inward and outward energies given in Definition \ref{energies} and give energy flux formula of them. 

\subsection{General Energy Flux Formula}

The following is a simple application of Green's Theorem, but it is indeed a general version of energy flux formula.
\begin{proposition}[General Energy Flux] \label{energy flux formula}
 Let $\Omega$ be a closed region in the right half $(0,\infty)\times \Rm$ of $r-t$ space. Its boundary $\Gamma$ consists of finite line segments, which are paralleled to either $t$-axis, $r$-axis or characteristic lines $t\pm r = 0$, and is oriented counterclockwise. Then we have an identity
 \begin{align}
  \pi \int_{\Gamma} \left(|w_r+w_t|^2 + \frac{2}{p+1}\cdot \frac{|w|^{p+1}}{r^{p-1}}\right) dr & + \left(|w_r+w_t|^2 - \frac{2}{p+1}\cdot \frac{|w|^{p+1}}{r^{p-1}}\right) dt \nonumber \\
  & \qquad -\frac{2\pi(p-1)}{p+1} \iint_{\Omega} \frac{|w|^{p+1}}{r^p} dr dt = 0. \label{energy flux inward}\\
  \pi \int_{\Gamma} \left(|w_r-w_t|^2 + \frac{2}{p+1}\cdot \frac{|w|^{p+1}}{r^{p-1}}\right) dr & + \left(-|w_r-w_t|^2 + \frac{2}{p+1}\cdot \frac{|w|^{p+1}}{r^{p-1}}\right) dt \nonumber \\
  &\qquad  + \frac{2\pi(p-1)}{p+1} \iint_{\Omega} \frac{|w|^{p+1}}{r^p} dr dt = 0 \label{energy flux outward}. 
 \end{align}
Furthermore, there exists a finite, nonnegative, continuous\footnote{Continuity means the function $\mu((-\infty,t])$ is a continuous function of $t$.} measure $\mu$ with $\mu(\Rm) \lesssim_p E$, which is solely determined by $u$ and independent of $\Omega$, so that the identities above also hold for regions $\Omega$ with part of its boundary on the $t$-axis. In this case the line integral from the point $(0,t_2)$ downward to $(0,t_1)$ along $t$-axis is understood as $- \pi \int_{t_1}^{t_2} 1 \,d\mu(t)$ in identity \eqref{energy flux inward} or $\pi \int_{t_1}^{t_2} 1 \,d\mu(t)$ in identity \eqref{energy flux outward}
\end{proposition}
\paragraph{Line integrals} Before we give the outline of the proof, let us first have a look at what the line integrals look like for different types of boundary line segments. In table \ref{line integrals in energy flux} we always assume $r_1<r_2$, $t_1<t_2$. For example, when we say a line segment $r=r_0$ goes downward, it starts at time $t_2$ and ends at time $t_1$.
\begin{table}[h]
\caption{Line integrals in energy flux}
\begin{center}
\begin{tabular}{|c|c|c|}\hline
 Boundary type & Inward Energy Case & Outward Energy Case\\
 \hline
 Horizontally $\rightarrow$ & $\int_{r_1}^{r_2}\left(|w_r+w_t|^2 + \frac{2}{p+1}\cdot \frac{|w|^{p+1}}{r^{p-1}}\right) dr$ & 
 $\int_{r_1}^{r_2}\left(|w_r-w_t|^2 + \frac{2}{p+1}\cdot \frac{|w|^{p+1}}{r^{p-1}}\right) dr$ \\
 \hline
 Horizontally $\leftarrow$ & $-\int_{r_1}^{r_2}\left(|w_r+w_t|^2 + \frac{2}{p+1}\cdot \frac{|w|^{p+1}}{r^{p-1}}\right) dr$ & 
 $-\int_{r_1}^{r_2}\left(|w_r-w_t|^2 + \frac{2}{p+1}\cdot \frac{|w|^{p+1}}{r^{p-1}}\right) dr$ \\
 \hline
 $r=r_0>0$,$\uparrow$ & $\int_{t_1}^{t_2} \left(|w_r+w_t|^2 - \frac{2}{p+1}\cdot \frac{|w|^{p+1}}{r^{p-1}}\right) dt$ &
 $\int_{t_1}^{t_2} \left(-|w_r-w_t|^2 + \frac{2}{p+1}\cdot \frac{|w|^{p+1}}{r^{p-1}}\right) dt$ \\
 \hline
 $r=r_0>0$,$\downarrow$ & $\int_{t_1}^{t_2} \left(-|w_r+w_t|^2 + \frac{2}{p+1}\cdot \frac{|w|^{p+1}}{r^{p-1}}\right) dt$ &
 $\int_{t_1}^{t_2} \left(|w_r-w_t|^2 - \frac{2}{p+1}\cdot \frac{|w|^{p+1}}{r^{p-1}}\right) dt$ \\
 \hline
 $r=0$, $\downarrow$ & $-\int_{t_1}^{t_2} 1 d\mu(t)$ & $+\int_{t_1}^{t_2} 1 d\mu(t)$ \\
 \hline
 $t+r=s$, $\searrow$ & $\frac{4}{p+1} \int_{t_1}^{t_2} \frac{|w(s-t,t)|^{p+1}}{(s-t)^{p-1}} dt$ & $2\int_{t_1}^{t_2} |w_r(s\!-\!t,t)-w_t(s\!-\!t,t)|^2 dt$\\
 \hline
 $t+r=s$, $\nwarrow$ & $-\frac{4}{p+1} \int_{t_1}^{t_2} \frac{|w(s-t,t)|^{p+1}}{(s-t)^{p-1}} dt$ & $-2\int_{t_1}^{t_2} |w_r(s\!-\!t,t)-w_t(s\!-\!t,t)|^2 dt$\\
 \hline
 $t-r=\tau$, $\nearrow$ & $2\int_{t_1}^{t_2} |w_r(t\!-\!\tau,t)+w_t(t\!-\!\tau,t)|^2 dt$ & $\frac{4}{p+1} \int_{t_1}^{t_2} \frac{|w(t-\tau,t)|^{p+1}}{(t-\tau)^{p-1}} dt$\\
 \hline
 $t-r=\tau$, $\swarrow$ & $-2\int_{t_1}^{t_2} |w_r(t\!-\!\tau,t)+w_t(t\!-\!\tau,t)|^2 dt$ & $-\frac{4}{p+1} \int_{t_1}^{t_2} \frac{|w(t-\tau,t)|^{p+1}}{(t-\tau)^{p-1}}dt$\\
 \hline
\end{tabular}
\end{center}
\label{line integrals in energy flux}
\end{table}
\paragraph{Physical Meaning} The physical meaning of the flux across characteristic lines, which corresponds to light cone in $\Rm^3$, and other related terms, are given as 
\begin{itemize}
 \item The terms in the form of $\int_{t_1}^{t_2} |w_r\pm w_t|^2 dt$ are the amount of energy which moves across characteristic lines due to liner wave propagation. 
 \item The terms in the form of $\int_{t_1}^{t_2} \frac{|w|^{p+1}}{r^{p-1}} dt$ are the amount of energy which moves across characteristic lines due to nonlinear effect.
 \item The terms in the form of $\int_{t_1}^{t_2} 1 d\mu(t)$ are the amount of energy which moves across the line $r=0$, thus transforms from inward energy to outward energy during the given period of time. 
 \item The double integral in the identities is the amount of energy which transform from inward energy to outward energy due to nonlinear effect in the given space-time region.   
\end{itemize}
Now let us give a outline of proof for Proposition \ref{energy flux formula}
\begin{proof} 
 Without loss of generality let us prove the first identity in the proposition. The second one can be prove in the same way. If the region is away from the $t$-axis, then we only need to apply Green's formula on the line integrals of the given vector fields and use the equation $w_{tt}-w_{rr} = - \frac{|w|^{p-1}w}{r^{p-1}}$. In the process of calculation the second derivatives are involved, thus we apply smooth approximation techniques when necessary. Now let us consider the case when part of the boundary is on the $t$-axis. In this case a limiting process $r \rightarrow 0^+$ is required. It suffices to show there exists a nondecreasing continuous function $P(t)$ with $P(+\infty)-P(-\infty) \lesssim_p E$ so that for all $t_1<t_2$, we have 
\begin{equation} \label{def of chi}
 \lim_{r\rightarrow 0^+} \int_{t_1}^{t_2} \left(|w_r(r,t)+w_t(r,t)|^2 - \frac{2}{p+1}\cdot \frac{|w(r,t)|^{p+1}}{r^{p-1}}\right) dt = P(t_2)-P(t_1).
\end{equation} 
We first choose $\Omega = [r,1]\times [t_1,t_2]$, which is away from the $t$-axis, thus we can use formula \eqref{energy flux inward} and write
\begin{align}
 - \int_{t_1}^{t_2} \left(|w_r(r,t)+w_t(r,t)|^2  - \frac{2}{p+1}\cdot \frac{|w(r,t)|^{p+1}}{r^{p-1}}\right) dt +  E_-(t_1;r,1) & \nonumber \\
  +\int_{t_1}^{t_2} \left(|w_r(1,t)+w_t(1,t)|^2  -\frac{2}{p+1}\cdot |w(1,t)|^{p+1}\right) dt - E_-(t_2;r,1)&  \label{identity rectangle 1} \\
   -\frac{2\pi(p-1)}{p+1} \int_{t_1}^{t_2} \int_r^1 \frac{|w(r',t)|^{p+1}}{(r')^p} dr' dt & = 0 \nonumber
\end{align}
All the other four terms in the identity above converge as $r \rightarrow 0^+$, therefore we know the limit in \eqref{def of chi} converges for any time $t_1<t_2$. Let us define 
\[
 P(t) = \lim_{r\rightarrow 0^+} \int_{0}^{t} \left(|w_r(r,t')+w_t(r,t')|^2 - \frac{2}{p+1}\cdot \frac{|w(r,t')|^{p+1}}{r^{p-1}}\right) dt'.
\]
This guarantees that identity \eqref{def of chi} always holds. By \eqref{identity rectangle 1} we also have 
\begin{align}
 P(t) = & E_-(0;0,1) - E_-(t;0,1) + \int_{0}^{t} \left(|w_r(1,t')+w_t(1,t')|^2 -\frac{2}{p+1}\cdot |w(1,t')|^{p+1}\right) dt' \nonumber\\
 & -\frac{2\pi(p-1)}{p+1} \int_{0}^{t} \int_0^1 \frac{|w(r',t')|^{p+1}}{(r')^p} dr' dt'. \label{expression of P1}
\end{align}
We claim that $E_-(t;0,1) = \pi \int_0^1 \left(|w_r(r,t)+w_t(r,t)|^2 + \frac{2}{p+1}\cdot \frac{|w(r,t)|^{p+1}}{r^{p-1}}\right)dr$ is a continuous function of $t$ by the following facts:
\begin{itemize}
  \item $(u(\cdot,t),u_t(\cdot,t)) \in C(\Rm_t; \dot{H}^1 \times L^2(\Rm^3)) \Rightarrow (w_r(\cdot,t),w_t(\cdot,t)) \in C(\Rm_t; L^2 \times L^2)$;
  \item $4\pi \int_0^1 \frac{|w(r,t)|^{p+1}}{r^{p-1}} dr = \int_{B(\mathbf{0},1)} |u(x,t)|^{p+1} dx$;
  \item The Sobolev embedding $\dot{H}^1(\Rm^3) \hookrightarrow L^{p+1}(B({\mathbf 0},1))$;
\end{itemize} 
Each term in the right hand side of \eqref{expression of P1} is a continuous function of $t$, thus $P(t)$ is also a continuous function. Furthermore, we can combine the inequality $w(r,t) \lesssim_E r^{1/2}$ and Corollary \ref{cor w morawetz} to obtain
\begin{equation}
 \int_{-\infty}^{\infty} \frac{|w(r,t)|^{p+1}}{r^{p-1}} dt = \int_{-\infty}^{\infty} \frac{|w(r,t)|^{p-1} r^2 |u(r,t)|^2}{r^{p-1}} dt\lesssim_{p,E} r^{\frac{5-p}{2}} \int_{-\infty}^\infty |u(r,t)|^2 dt \lesssim_{p,E} r^{\frac{5-p}{2}}. \label{decay of w line}
\end{equation}
Thus we can redefine 
\[
 P(t) = \lim_{r\rightarrow 0^+} \int_{0}^{t} \left(|w_r(r,t')+w_t(r,t')|^2 \right) dt'.
\]
This means $P(t)$ is also a non-decreasing function with $P(0)=0$. By the monotonicity and continuity of $P(t)$ we can always define a Borel measure $\mu$ by $\mu((a,b))=P(b)-P(a)$. This enable us to rewrite \eqref{def of chi} into
\[
  \lim_{r\rightarrow 0^+} \int_{t_1}^{t_2} \left(|w_r(r,t)+w_t(r,t)|^2 - \frac{2}{p+1}\cdot \frac{|w(r,t')|^{p+1}}{r^{p-1}}\right) dt = \int_{t_1}^{t_2} 1 d\mu(t).
\]
Please note that it is unnecessary to specify whether the integral $\int_{t_1}^{t_2} 1d\mu(t)$ include the endpoints or not. Because the continuity of $P(t)$ implies $\mu(\{t_0\})=0$ for any single point set $\{t_0\}$. Finally we use the second inequality of Corollary \ref{cor w morawetz} to deduce $\mu(\Rm) \lesssim_p E$.
\end{proof}
\begin{remark}
 If the radial solution $u(x,t)$ is sufficiently smooth at the origin, then a simple calculation shows that $d\mu(t) = |u(0,t)|^2 dt$. In fact, we have 
 \begin{equation} \label{limit at origin smooth}
  \lim_{r\rightarrow 0^+} |w_r(r,t)+w_t(r,t)|^2 = \lim_{r\rightarrow 0^+} |w_r(r,t)-w_t(r,t)|^2 = |u(0,t)|^2
 \end{equation}
 If we follow the same procedure as above to prove the energy flux formula for outward energy, then the limiting process will give the same measure $\mu$ because of \eqref{limit at origin smooth} and smooth approximation techniques. Another way to prove the energy flux formula for outward energy is to write $E_+ = E-E_-$.
\end{remark}
\begin{remark}
 We can also prove $d\mu(t) = |\xi(t)|^2 dt$. Here $\xi(t) \in L^2(\Rm)$. We prefer $|\xi(t)|^2$ to a single $L^1$ function, because when $u$ is sufficiently smooth, we have $d\mu(t) = |u(0,t)|^2 dt$, as mentioned above. Since the proof uses elements from later sections, we put the proof in Appendix for convenience. 
\end{remark}
\paragraph{Notation for Flux} For convenience we write the energy flux across a characteristic line $t\pm r = \hbox{Const}$, which corresponds to a light cone in $\Rm^3$ in the following way.

\begin{definition}[Notations of Energy flux]
 Given $s,\tau \in \Rm$, we define full energy flux
\begin{align*} 
 Q_-^- (s) & =  \frac{4\pi}{p+1} \int_{-\infty}^{s} \frac{|w(s-t,t)|^{p+1}}{(s-t)^{p-1}} dt;\\
 Q_+^-(s) & = 2\pi \int_{-\infty}^{s} |w_r(s-t,t)-w_t(s-t,t)|^2 dt;\\
 Q_-^+(\tau) & = 2\pi \int_{\tau}^{+\infty} |w_r(t-\tau,t)+w_t(t-\tau,t)|^2 dt;\\
 Q_+^+(\tau) & = \frac{4\pi}{p+1} \int_{\tau}^{+\infty} \frac{|w(t-\tau,t)|^{p+1}}{(t-\tau)^{p-1}} dt.
\end{align*}
A negative sign upper index means that this is energy flux across the characteristic line $t+r = s$; otherwise this is energy flux across the characteristic line $t-r = \tau$. The lower index indicates whether this is energy flux of inward energy ($-$) or outward energy ($+$). We can also consider their truncated version, which is the energy flux across a line segment of a characteristic line.  
\begin{align*} 
 Q_-^- (s; t_1,t_2) & =  \frac{4\pi}{p+1} \int_{t_1}^{t_2} \frac{|w(s-t,t)|^{p+1}}{(s-t)^{p-1}} dt,& &t_1<t_2\leq s;\\
 Q_+^-(s; t_1,t_2) & = 2\pi \int_{t_1}^{t_2} |w_r(s-t,t)-w_t(s-t,t)|^2 dt,& &t_1<t_2\leq s; \\
 Q_-^+(\tau; t_1,t_2) & = 2\pi \int_{t_1}^{t_2} |w_r(t-\tau,t)+w_t(t-\tau,t)|^2 dt,& &\tau\leq t_1<t_2;\\
 Q_+^+(\tau; t_1,t_2) & = \frac{4\pi}{p+1} \int_{t_1}^{t_2} \frac{|w(t-\tau,t)|^{p+1}}{(t-\tau)^{p-1}} dt,& &\tau\leq t_1<t_2.
\end{align*}
\end{definition}
\subsection{Energy Flux Formula for Triangle}
We can apply our general energy flux formula for a triangle region. This will be frequently used in Section \ref{sec:distribution}. 
\begin{proposition}[Triangle Law] \label{triangle law}
Given any $s_0>t_0$, we can define $\Omega = \{(r,t):t>t_0, r>0, r+t<s_0\}$ and write
\[
 E_-(t_0;0,s_0-t_0) = \pi \int_{t_0}^{s_0} 1\mu(t) + Q_-^-(s_0;t_0,s_0) + \frac{2\pi(p-1)}{p+1} \iint_{\Omega} \frac{|w(r,t)|^{p+1}}{{r}^p} dr dt.
\]
 We can substitute $s_0$ by $t_0+r_0$ with $r_0>0$ and rewrite this into the form
\[
 E_-(t_0;0,r_0) = \pi \int_{t_0}^{t_0+r_0} 1\mu(t) + Q_-^-(t_0+r_0;t_0,t_0+r_0) + \frac{2\pi(p-1)}{p+1} \iint_{\Omega} \frac{|w(r,t)|^{p+1}}{{r}^p} dr dt.
\]
\end{proposition}

\section{Energy Distribution of Solutions} \label{sec:distribution}

\subsection{Limits of Inward and Outward Energies}

All inward and outward energies are clearly bounded above by the full energy $E$, since $E= E_-+E_+$. This immediately gives 
\begin{lemma}[Boundedness of Energy Flux] \label{bounded of flux}
 We have a uniform upper bound for all $s, \tau \in \Rm$,
 \begin{align*}
  Q_{-}^{-}(s), Q_+^+ (\tau), \pi \int_{-\infty}^{+\infty} 1 d\mu(t) \leq E.
 \end{align*}
\end{lemma}
\begin{proof}
We apply triangle law (Proposition \ref{triangle law}) on the region $\Omega(s,t_0) = \{(r,t): r+t\leq s, r>0, t>t_0\}$ for any $t_0<s$ and obtain
\[
 E_-(t_0;0,s-t_0) = \pi \int_{t_0}^s 1 d\mu(t) + Q_-^-(s;t_0,s) + \frac{2(p-1)\pi}{p+1} \int_{\Omega(s,t_0)} \frac{|w(r,t)|^{p+1}}{r^p} dr dt,
\]
as shown in figure \ref{figure finiteQ}. Making $t_0 \rightarrow -\infty$ we have 
\[
 \pi \int_{-\infty}^s 1 d\mu(t) + Q_-^-(s) + \frac{2(p-1)\pi}{p+1} \iint_{r>0, r+t\leq s} \frac{|w(r,t)|^{p+1}}{r^p} dr dt \leq E.
\]
Thus we have $Q_-^-(s) \leq E$ for any $s \in \Rm$. Finally we let $s \rightarrow +\infty$ in the inequality above to obtain 
$\pi \int_{-\infty}^{\infty} 1 d\mu(t) \leq E$. The outward energy flux $Q_+^+$ can be dealt with in the same way.
\end{proof}

 \begin{figure}[h]
 \centering
 \includegraphics[scale=0.9]{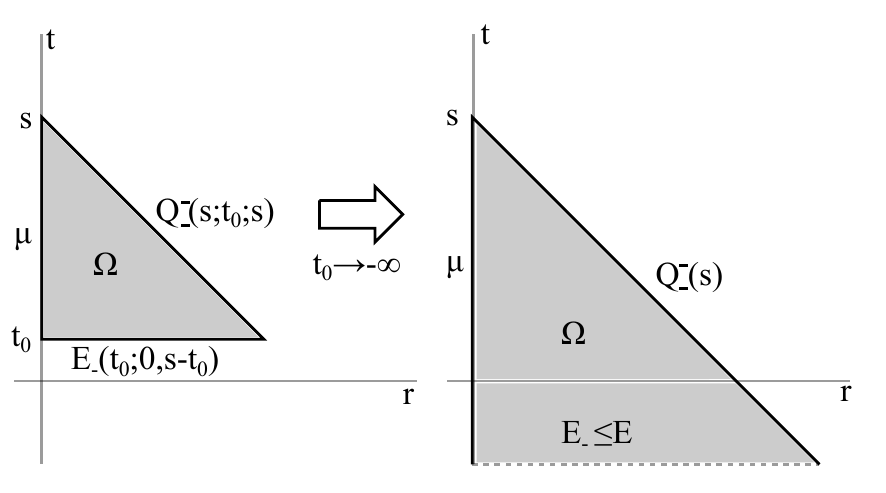}
 \caption{Illustration for proof of Lemma \ref{bounded of flux}} \label{figure finiteQ}
\end{figure}
\noindent Before we consider the monotonicity and asymptotic behaviour of inward and outward energy as $t \rightarrow \pm \infty$, let us first give a technical lemma.

\begin{lemma} \label{lift of r}
 Given any $t_0\in \Rm$, we have
 \[
  \liminf_{r\rightarrow +\infty} Q_-^-(t_0+r,t_0,t_0+r) = 0.
 \]
\end{lemma}
\begin{proof}
Fix $r_0>0$. We consider the following integral and apply the change of variable $\bar{r} = t_0+r-t$
\begin{align*}
 \int_{r_0}^{2r_0} Q_{-}^{-} (t_0+r,t_0,t_0+r) dr & = \frac{4\pi}{p+1} \int_{r_0}^{2r_0} \int_{t_0}^{t_0+r} \frac{|w(t_0+r-t,t)|^{p+1}}{(t_0+r-t)^{p-1}} dt dr \\
 & = \frac{4\pi}{p+1} \iint_{\Omega(r_0)} \frac{|w(\bar{r},t)|^{p+1}}{\bar{r}^{p-1}} d\bar{r} dt\\
 & \leq \frac{8\pi r_0}{p+1} \iint_{\Omega(r_0)} \frac{|w(\bar{r},t)|^{p+1}}{\bar{r}^p} d\bar{r} dt
 \end{align*}
 Here $\Omega(r_0) = \{(\bar{r},t):\bar{r}>0,t>t_0,t_0+r_0<t+\bar{r}<t_0+2r_0\} \subset [0,2r_0]\times \Rm$, as shown in figure \ref{upperQ}.
 \begin{figure}[h]
 \centering
 \includegraphics[scale=0.85]{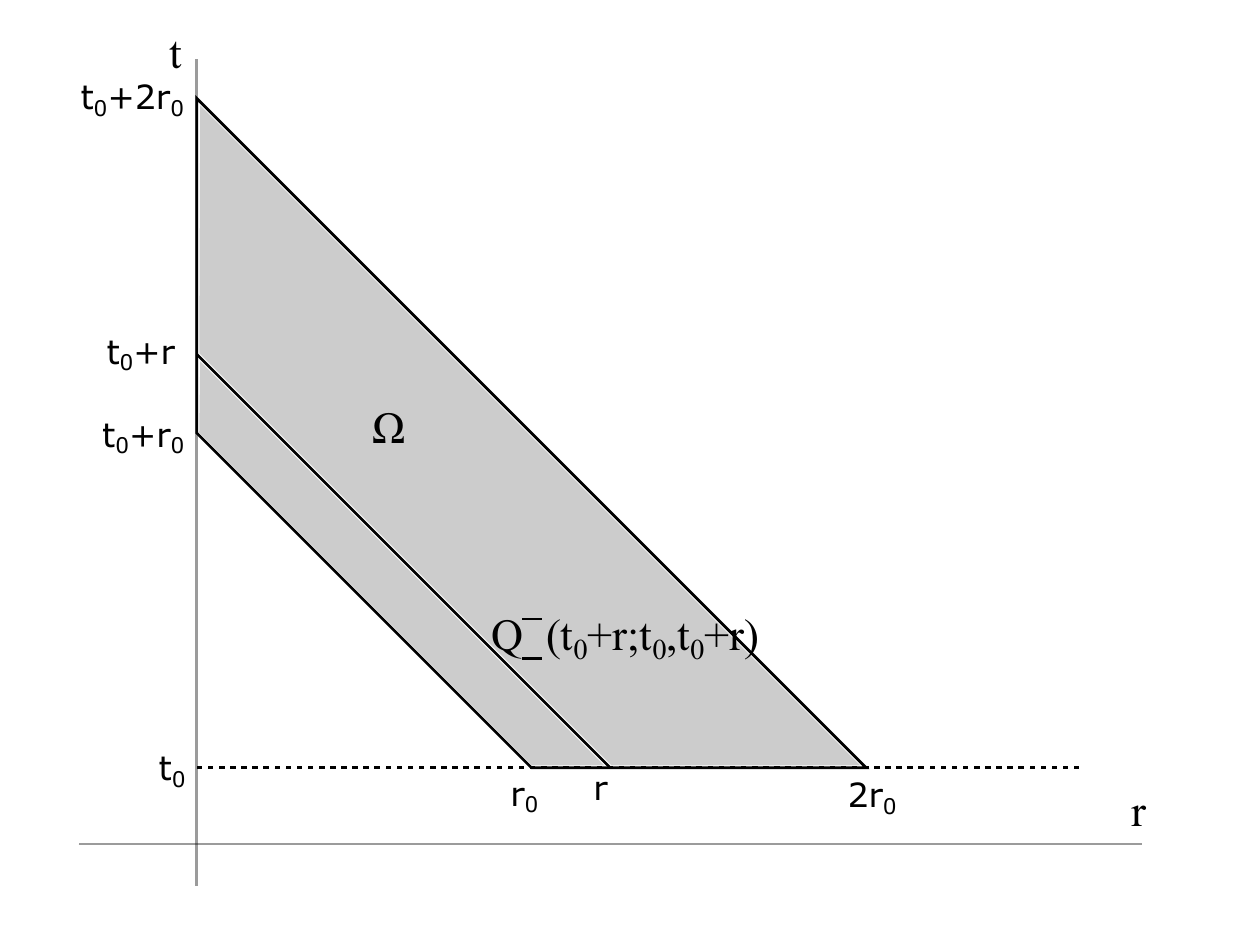}
 \caption{Illustration of integral region} \label{upperQ}
\end{figure}
Now we are able to apply the mean value theorem to conclude there exists a number $r \in [r_0,2r_0]$ so that
\[
 Q_{-}^{-} (t_0+r,t_0,t_0+r) \leq \frac{8\pi}{p+1} \iint_{\Omega(r_0)} \frac{|w(\bar{r},t)|^{p+1}}{\bar{r}^p} d\bar{r} dt.
\]
If we make $r_0 \rightarrow +\infty$ in the argument above and observe the fact
\[
 \lim_{r_0 \rightarrow + \infty} \iint_{\Omega(r_0)} \frac{|w(\bar{r},t)|^{p+1}}{\bar{r}^p} d\bar{r} dt = 0,
\] 
we obtain the lower limit as desired.
\end{proof}

\noindent Now we have 

\begin{proposition}[Monotonicity of Inward and Outward Energies] \label{monotonicity of energies}
 The inward energy $E_-(t)$ is a decreasing function of $t$, while the outward energy $E_+(t)$ is an increasing function of $t$. In addition 
\begin{align*}
 &\lim_{t\rightarrow +\infty} E_-(t) = 0; & &\lim_{t\rightarrow -\infty} E_+(t) = 0.&
\end{align*}
\end{proposition}
\begin{proof}
Let us prove the monotonicity and limit of inward energy. The outward energy can be dealt with in the same way. First of all, we apply inward energy flux formula on the region $\Omega = \{(r,t): r>0,t_1<t<t_2, r+t<s\}$ for $s\geq t_2>t_1$, as shown in the upper half of figure \ref{traptri}.
\begin{align}
 & E_-(t_2; 0, s-t_2) - E_-(t_1;0,s-t_1) \nonumber \\
 & \quad = -\pi \int_{t_1}^{t_2} 1 d\mu(t) - Q_-^-(s;t_1,t_2) - \frac{2\pi (p-1)}{p+1} \iint_{\Omega} \frac{|w(r,t)|^{p+1}}{r^p} dr dt. \label{trapezoid app 1}
\end{align}
Now let us recall the inequality $|w(r,t)| \lesssim_{p,E} r^{(p-1)/(p+3)}$ given by Lemma \ref{best upper bound pointwise}. This implies $Q_-^-(s;t_1,t_2) \rightarrow 0$ as $s\rightarrow \infty$. Therefore we can make $s\rightarrow \infty$ in the identity above and obtain
\[
 E_-(t_2) - E_-(t_1) = -\pi \int_{t_1}^{t_2} 1 d\mu(t) - \frac{2\pi (p-1)}{p+1} \int_{t_1}^{t_2} \int_0^\infty \frac{|w(r,t)|^{p+1}}{r^p} dr dt < 0.
\]
This gives the monotonicity. Next we apply triangle law with $t_0\in \Rm$ and $r_0>0$, as shown in the lower part of figure \ref{traptri}
\begin{align*}
 E_-(t_0;0,r_0) & = \pi \int_{t_0}^{t_0+r_0} 1 d\mu(t) + Q_-^-(t_0+r_0; t_0, t_0+r_0) \\
 & \qquad + \frac{2\pi (p-1)}{p+1} \int_{t_0}^{t_0+r_0} \int_0^{t_0+r_0-t} \frac{|w(r,t)|^{p+1}}{r^p} dr dt.
\end{align*}
According to Lemma \ref{lift of r}, we can take a limit $r_0 \rightarrow +\infty$ and write
\begin{equation}
 E_-(t_0) \leq \pi \int_{t_0}^{\infty} 1 d\mu(t) + \frac{2\pi (p-1)}{p+1} \int_{t_0}^{+\infty} \int_0^{+\infty} \frac{|w(r,t)|^{p+1}}{r^p} dr dt. \label{expression of E neg temp}
\end{equation}
Finally we can make $t_0\rightarrow +\infty$ and finish the proof.
\end{proof}

 \begin{figure}[h]
 \centering
 \includegraphics[scale=0.95]{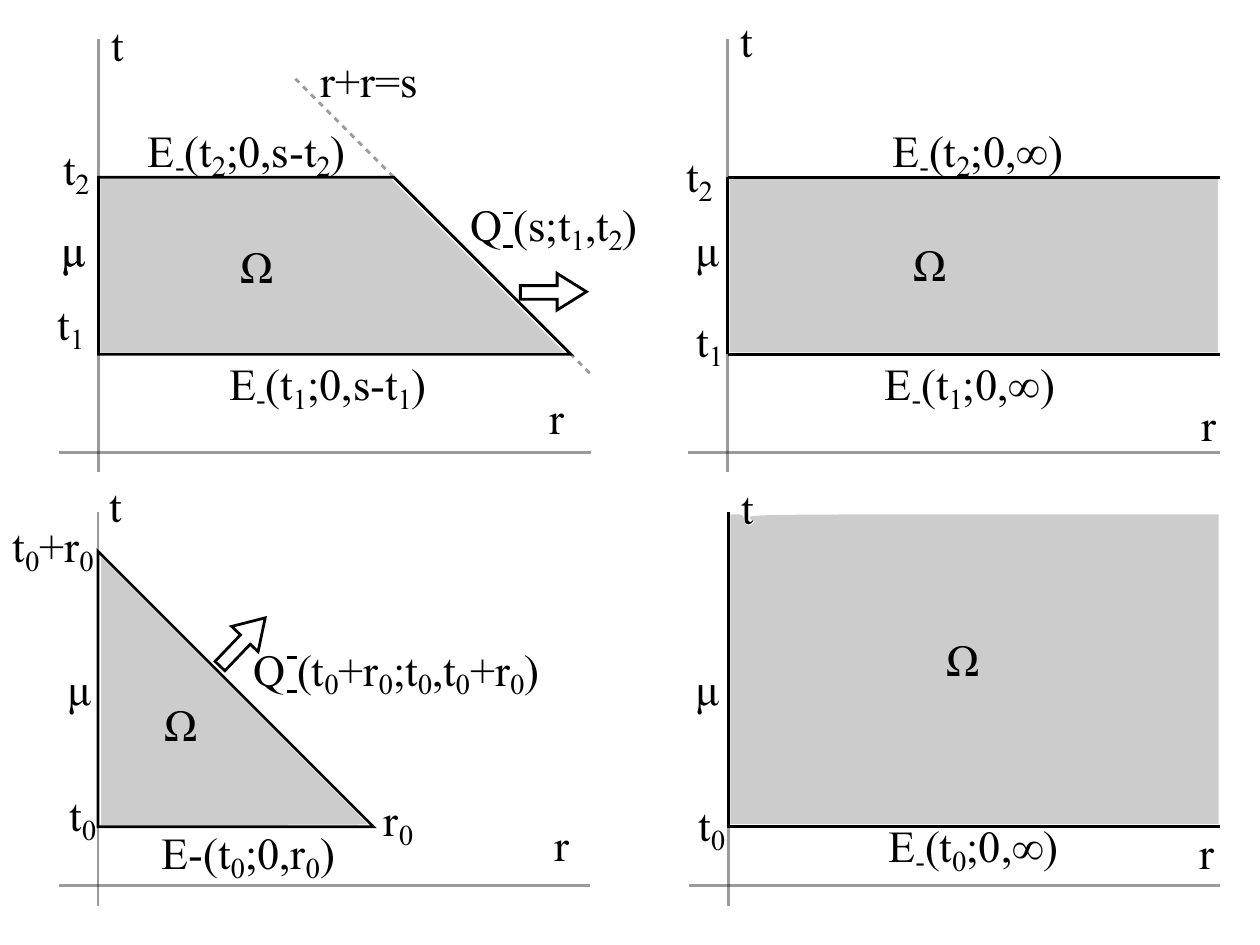}
 \caption{Illustration for proof of Proposition \ref{monotonicity of energies}} \label{traptri}
\end{figure}

\noindent This immediately gives
\begin{corollary} \label{asymptotic behaviour 2}
 We have the limits 
 \begin{align*}
  &\lim_{t\rightarrow -\infty} E_-(t) = E; & && &\lim_{t\rightarrow +\infty} E_+(t) = E.&\\
  &\lim_{t\rightarrow \pm \infty} \int_{0}^\infty \frac{|w(r,t)|^{p+1}}{r^{p-1}} dr = 0& &\Leftrightarrow &
  &\lim_{t\rightarrow \pm \infty} \int_{\Rm^3} |u(x,t)|^{p+1} dx = 0.&
 \end{align*}
\end{corollary}
\noindent We also have asymptotic behaviour of energy flux
\begin{proposition} \label{a limit of Q}
 We have the limits
 \begin{align*}
  &\lim_{s \rightarrow +\infty} Q_-^-(s) = 0;& &\lim_{\tau \rightarrow -\infty} Q_+^+(\tau) = 0.&
 \end{align*}
\end{proposition}
\begin{proof}
 Again we only give the proof for inward energy flux. An application of inward energy flux on the parallelogram $\Omega = \{(r,t): t_0<t<s, s<r+t<s'\}$ with $t_0<s<s'$ gives (Please see figure \ref{limitQ})
\begin{align}
 & E_-(s; 0, s'-s) - E_-(t_0;s-t_0,s'-t_0) \nonumber \\
 &\qquad  = Q_-^-(s;t_0,s)- Q_-^-(s';t_0,s) - \frac{2\pi (p-1)}{p+1} \int_{t_0}^s \int_{s-t}^{s'-t} \frac{|w(r,t)|^{p+1}}{r^p} dr dt. \nonumber
\end{align} 
Making $s'\rightarrow \infty$ we can discard the term $Q_-^-(s';t_0,s)$ as in the proof of Proposition \ref{monotonicity of energies} and rewrite the identity above into
\begin{align*}
 E_-(s) + \frac{2\pi (p-1)}{p+1} \int_{t_0}^s \int_{s-t}^{\infty} \frac{|w(r,t)|^{p+1}}{r^p} dr dt = E_-(t_0; s-t_0,\infty) + Q_-^-(s;t_0,s).
\end{align*}
Next we can take a limit as $t_0\rightarrow - \infty$.
\[
 E_-(s) + \frac{2\pi (p-1)}{p+1} \int_{-\infty}^s \int_{s-t}^{\infty} \frac{|w(r,t)|^{p+1}}{r^p} dr dt = \lim_{t\rightarrow -\infty} E_-(t;s-t,\infty) + Q_-^-(s).
\]
Finally we observe that both terms in left hand  converges to zero as $s \rightarrow +\infty$ and finish the proof. 
\end{proof}

 \begin{figure}[h]
 \centering
 \includegraphics[scale=1.2]{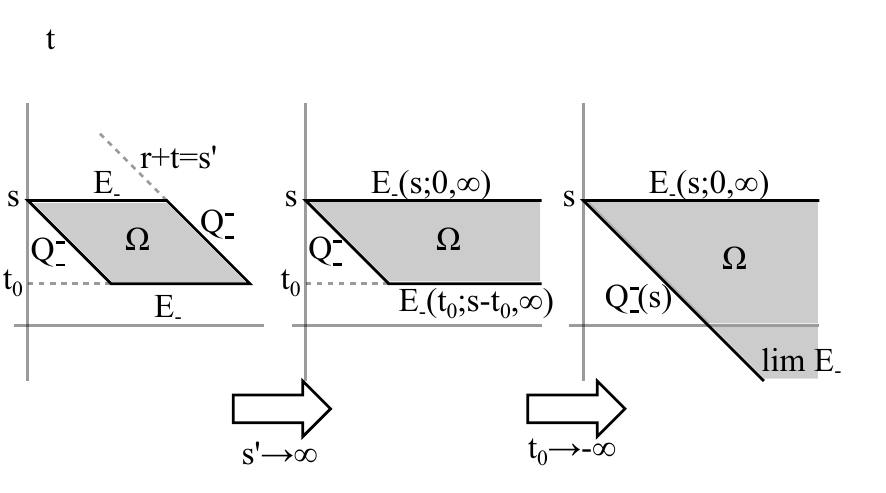}
 \caption{Illustration for proof of Proposition \ref{a limit of Q}} \label{limitQ}
\end{figure}

\begin{remark} \label{expression of E neg}
 This asymptotic behaviour of $Q_-^-$ means that the inequality \eqref{expression of E neg temp} is actually an identity as below, since now we have the limit of $Q_-^-$ vanishes rather than a lower limit.
 \begin{equation*}
 E_-(t) = \pi \int_{t}^{\infty} 1 d\mu(t') + \frac{2\pi (p-1)}{p+1} \int_{t}^{+\infty} \int_0^{+\infty} \frac{|w(r,t')|^{p+1}}{r^p} dr dt'. 
\end{equation*}
Making $t \rightarrow -\infty$ we have an identity
\begin{equation*}
 \pi \int_{-\infty}^{\infty} 1 d\mu(t) + \frac{2\pi (p-1)}{p+1} \int_{-\infty}^{+\infty} \int_0^{+\infty} \frac{|w(r,t)|^{p+1}}{r^p} dr dt = E.
\end{equation*}
This means that all the energy eventually transform from inward energy to outward energy by either passing through the origin or nonlinear effect when $t$ moves from $-\infty$ to $+\infty$.
\end{remark}

\subsection{Asymptotic Behaviour of $w_r\pm w_t$}

We can rewrite the equation $w_{tt}-w_{rr} = -\frac{|w|^{p-1}w}{r^{p-1}}$ into
\[
 \left(\partial_t -\partial_r\right)(\partial_t + \partial_r) w = - \frac{|w|^{p-1}w}{r^{p-1}}.
\]
Therefore we have 
\begin{proposition} \label{var of w partial}
Let $\tau< t_1<t_2<s$
\begin{align*}
 (w_r+w_t)(s-t_2,t_2) - (w_r+w_t)(s-t_1,t_1) & = -\int_{t_1}^{t_2} \frac{|w|^{p-1}w(s-t,t)}{(s-t)^{p-1}} dt.\\
 (w_r-w_t)(t_2-\tau,t_2) - (w_r-w_t)(t_1-\tau,t_1) & = +\int_{t_1}^{t_2} \frac{|w|^{p-1}w(t-\tau,t)}{(t-\tau)^{p-1}} dt.
\end{align*}
\end{proposition}
\noindent As a result, we can use the boundedness of energy flux $Q$ and obtain an estimate
\begin{proposition} \label{pointwise estimate ws and wtau}
 Let $\tau<t_1<t_2<s$
 \begin{align*}
 \left|(w_r+w_t)(s-t_2,t_2) - (w_r+w_t)(s-t_1,t_1)\right| & \lesssim_{p,E} (s-t_2)^{-\frac{p-2}{p+1}}\\
 \left|(w_r-w_t)(t_2-\tau,t_2) - (w_r-w_t)(t_1-\tau,t_1)\right| & \lesssim_{p,E} (t_1-\tau)^{-\frac{p-2}{p+1}}.
 \end{align*}
\end{proposition}
\begin{proof}
Both estimates are proved in the same way. Let us prove the first one. According to Proposition \ref{var of w partial}, it suffices to show 
\[
 I = \left|\int_{t_1}^{t_2} \frac{|w|^{p-1}w(s-t,t)}{(s-t)^{p-1}} dt\right| \lesssim_{p,E} (s-t_2)^{-\frac{p-2}{p+1}}
\]
This follows a straightforward calculation.
\begin{align*}
 I & \leq \int_{t_1}^{t_2} \frac{|w(s-t,t)|^{p}}{(s-t)^{p-1}} dt\\
 & \leq \left(\int_{t_1}^{t_2} \left(\frac{|w(s-t,t)|^p}{(s-t)^{\frac{(p-1)p}{p+1}}}\right)^{\frac{p+1}{p}}dt\right)^{\frac{p}{p+1}} 
 \left(\int_{t_1}^{t_2}\left(\frac{1}{(s-t)^{\frac{p-1}{p+1}}}\right)^{p+1} dt\right)^{\frac{1}{p+1}}\\
 & \lesssim_p \left(Q_-^-(s;t_1,t_2)\right)^{\frac{p}{p+1}}\left[(s-t_2)^{2-p}-(s-t_1)^{2-p}\right]^{\frac{1}{p+1}}\\
 & \lesssim_{p,E} (s-t_2)^{-\frac{p-2}{p+1}}.
\end{align*}
\end{proof}

\paragraph{The limit of $w_r \pm w_t$} By Proposition \ref{pointwise estimate ws and wtau}, we immediately have pointwise limits
\begin{align*}
  &\lim_{t \rightarrow -\infty} (w_r+w_t)(s-t,t) = g_-(s);& &\lim_{t \rightarrow +\infty} (w_r-w_t)(t-\tau,t) = g_+(\tau).&
\end{align*}
By Fatou's Lemma, we have $\|g_-\|_{L^2(\Rm)}, \|g_+\|_{L^2(\Rm)} \leq E/\pi$. Let us define 
\begin{align*}
 &\tilde{E}_- = \pi \int_\Rm |g_-(s)|^2 ds;& &\tilde{E}_+ = \pi \int_\Rm |g_+(\tau)|^2 d\tau.&
\end{align*} 
Energy flux of full energy gives us 
\begin{align*}
 & E(t;s-t,+\infty) \leq E(0;s,\infty), t<0,s>0& &\Rightarrow& &\lim_{s\rightarrow +\infty} \sup_{t<0} \int_{s-t}^\infty |w_r(r,t)+w_t(r,t)|^2 dt = 0;&\\
 & E(t;t-\tau,+\infty) \leq E(0;-\tau,\infty),t>0, \tau<0& &\Rightarrow& &\lim_{\tau\rightarrow -\infty} \sup_{t>0} \int_{t-\tau}^\infty |w_r(r,t)-w_t(r,t)|^2 dt = 0.&
\end{align*}
As a result, we have 
\begin{proposition} \label{limits of wr pm wt}
There exist functions $g_-(s), g_+(\tau)$ with $\|g_-\|_{L^2(\Rm)}^2, \|g_+\|_{L^2(\Rm)}^2 \leq E/\pi$. so that 
\begin{align*}
 &\left|(w_r+w_t)(s-t,t) - g_-(s)\right| \lesssim_{p,E} (s-t)^{-\frac{p-2}{p+1}},& &t<s;&\\
 &\left|(w_r-w_t)(t-\tau,t) - g_+(\tau)\right| \lesssim_{p,E} (t-\tau)^{-\frac{p-2}{p+1}},& &t>\tau.&
\end{align*}
Furthermore we have the following $L^2$ convergence for any $s_0,\tau_0 \in \Rm$
\begin{align*}
 &\lim_{t \rightarrow -\infty} \|(w_r+w_t)(s-t,t) - g_-(s)\|_{L_s^2([s_0,\infty))} = 0; \\
 &\lim_{t \rightarrow +\infty} \|(w_r-w_t)(t-\tau,t) - g_+(\tau)\|_{L_\tau^2((-\infty, \tau_0])} = 0.
\end{align*}
\end{proposition}

\paragraph{The scattering target} Now we can define 
\begin{align*}
 &V^-(r,t) = \frac{1}{2}\int_{t-r}^{t+r} g_-(s) ds;& &V^+(r,t) = \frac{1}{2}\int_{t-r}^{t+r} g_+(\tau) d\tau.&
\end{align*}
One can check $(V^\pm,V_t^\pm) \in C(\Rm_t; \dot{H}^1\times L^2)$ are one-dimensional free waves. We can also define 
\begin{align*}
 v^\pm (x,t) = V^\pm (|x|,t)/|x|
\end{align*}
as three-dimensional free waves with energy $\tilde{E}_\pm$. By Proposition \ref{limits of wr pm wt} and the fact $\displaystyle \lim_{t\rightarrow \pm \infty} E_\mp(t) = 0$ (given by Corollary \ref{asymptotic behaviour 2}) we can conduct a simple calculation and conclude for $s_0,\tau_0 \in \Rm$
\begin{align}
 & \lim_{t\rightarrow -\infty} \left\|\begin{pmatrix} w_r(\cdot,t)\\ w_t(\cdot,t) \end{pmatrix} -\begin{pmatrix} V_r^-(\cdot,t)\\ V_t^-(\cdot,t) \end{pmatrix} \right\|_{L^2([s_0-t,+\infty))} = 0. \label{outer estimate 1} \\
 & \lim_{t\rightarrow +\infty} \left\|\begin{pmatrix} w_r(\cdot,t)\\ w_t(\cdot,t) \end{pmatrix} -\begin{pmatrix} V_r^+(\cdot,t)\\ V_t^+(\cdot,t) \end{pmatrix} \right\|_{L^2([t-\tau_0,+\infty))} = 0. \nonumber
\end{align}
Rewriting this in term of the original solution $u$ and $v^\pm$ via identity \eqref{energy transformation}, we obtain Part (a) of Theorem \ref{main 1}. Next we prove Part (b), i.e. an equivalent condition for scattering.
\begin{proposition} \label{equivalent scattering}
 If $\tilde{E}_- = E$, then the solution $u$ scatters in the negative time direction 
 \[
  \lim_{t\rightarrow -\infty} \left\|\begin{pmatrix} u(\cdot,t)\\ u_t(\cdot,t) \end{pmatrix} -\begin{pmatrix} v^-(\cdot,t)\\ v_t^-(\cdot,t) \end{pmatrix} \right\|_{\dot{H}^1\times L^2(\Rm^3)} = 0.
 \]
\end{proposition}
\begin{proof}
 Given $\eps>0$, there exists a number $s_0 \in \Rm$, so that $\pi \int_{-\infty}^{s_0} |g_-(s)|^2 ds < \eps$. According to Proposition \ref{limits of wr pm wt}, for sufficiently large negative time $t<t_1$, we have 
\begin{equation} 
 E(t;s_0-t,\infty) \geq \pi \int_{s_0-t}^\infty \left|(w_r+w_t)(r,t)\right|^2 dr > E-\eps \Rightarrow E(t;0,s_0-t)<\eps. \label{remaining energy 1}
\end{equation}
By the definition of $V^-$ and our assumption on $s_0$ we also have
\begin{align*}
 & \pi \int_{0}^{s_0-t} \left|(V_r^- +V_t^-)(r,t)\right|^2 dr = \pi \int_{t}^{s_0} |g_-(s)|^2 ds < \eps,& &\hbox{for}\; t<s_0;&\\
 & \pi \int_{0}^{s_0-t} \left|(V_r^- -V_t^-)(r,t)\right|^2 dr = \pi \int_{2t-s_0}^{t} |g_-(s)|^2 ds \rightarrow 0,& &\hbox{as}\; t \rightarrow -\infty.&
\end{align*}
Therefore we have $2\pi \int_0^{s_0-t} \left(|V_r^-|^2 + |V_t^-|^2\right) dr < \eps$ for sufficiently small $t<t_2$. Combining this with \eqref{remaining energy 1} we have the following inequality for any $t<\min\{t_1,t_2\}$:
\[
 2\pi \int_0^{s_0-t} \left(|V_r^- -w_r|^2 + |V_t^- -w_r|^2\right) dr < 4\eps.
\]
Combining this with \eqref{outer estimate 1}, we know for sufficiently large negative $t$
\[
  2\pi \int_0^{+\infty} \left(|V_r^- -w_r|^2 + |V_t^- -w_r|^2\right) dr < 4\eps.
\]
We can rewrite this in terms of $u$ and $v^-$ by \eqref{energy between u and w} and finish the proof. 
\end{proof}

\paragraph{Remaining Energy} Proposition \ref{equivalent scattering} implies that if we did not have the scattering in the negative time direction, then the difference $E-\tilde{E}_-$ would be positive. Let us try to locate the remaining energy $E-\tilde{E}_-$. Proposition \ref{limits of wr pm wt} tells us that the location is inside the ball $B({\mathbf 0}, s_0-t)$ when $|t|$ is sufficiently large for any given $s_0$. In other words, this part of energy eventually enters and stays in any backward light cone as $t\rightarrow -\infty$. It travels slower than the light.  The following lemma, however, shows that its out-going speed is closed to the light speed. 
\begin{lemma} \label{lower bound of speed}
 Given $c\in(0,1)$, we have the limit
 \[
  \lim_{t \rightarrow \pm \infty} E_{\pm}(t;0,c|t|) = 0. 
 \]
\end{lemma} 
\begin{proof}
 Again we only need to give a proof for inward energy. First of all, we can apply the triangle law on the triangle region $\{(r',t'): t'>t, r'>0, t'+r'<t+r\}$ with 
 $c|t|<r<\frac{c+1}{2}|t|$. 
\begin{align*}
  E_-(t;0,r) & = \pi \int_{t}^{t+r} 1 d\mu(t') + Q_-^-(t+r; t, t+r) \\
 & \qquad + \frac{2\pi (p-1)}{p+1} \int_{t}^{t+r} \int_0^{t+r-t'} \frac{|w(r',t')|^{p+1}}{r'^p} dr' dt'.
\end{align*}
By the upper bound of $r$ we have $E_-(t;0,r) \leq P(t) + Q_-^-(t+r; t, t+r)$ where 
\[
 P(t) = \pi \int_{t}^{\frac{1-c}{2}t} 1 d\mu(t') + \frac{2\pi (p-1)}{p+1} \int_{t}^{\frac{1-c}{2}t} \int_0^{\frac{1-c}{2}t-t'} \frac{|w(r',t')|^{p+1}}{r'^p} dr' dt'
\]
satisfies $\displaystyle \lim_{t\rightarrow -\infty} P(t) = 0$. We can integrate the inequality above for $r \in \left(c|t|,\frac{c+1}{2}|t|\right)$ and obtain
\begin{align}
 \frac{1-c}{2}|t| \cdot E_{-}(t,0,c|t|) & \leq \int_{c|t|}^{\frac{c+1}{2}|t|} E_-(t,0,r) dr \nonumber\\
 & \leq \frac{1-c}{2}|t|P(t) + \int_{c|t|}^{\frac{c+1}{2}|t|} Q_-^-(t+r;t,t+r) dr. \label{upper bound of E ct}
\end{align}
Following the same argument in Lemma \ref{lift of r}, as shown in figure \ref{figure energydis}, we have 
\begin{align*}
 \int_{c|t|}^{\frac{c+1}{2}|t|} Q_-^-(t+r;t,t+r) dr & = \frac{4}{p+1} \iint_{\Omega(t)} \frac{|w(r',t')|^{p+1}}{r'^{p-1}} dr' dt'\\
 & \leq \frac{2(1+c)|t|}{p+1} \iint_{\Omega(t)} \frac{|w(r',t')|^{p+1}}{r'^p} dr' dt'.
\end{align*}
Here $\Omega(t) = \left\{(r',t'): t'>t,r'>0,(1-c)t<t'+r'<\frac{1-c}{2}t\right\}\subset \{(r',t'): r'\leq \frac{c+1}{2}|t|\}$. Plugging this upper bound in \eqref{upper bound of E ct} and dividing both sides by $\frac{1-c}{2}|t|$, we obtain
\[
 E_{-}(t,0,c|t|)  \leq P(t) + \frac{4(1+c)}{(p+1)(1-c)}  \iint_{\Omega(t)} \frac{|w(r',t')|^{p+1}}{r'^p} dr' dt'.
\]
Now we can take the limit $t\rightarrow -\infty$ on both sides and finish the proof.
\end{proof}

 \begin{figure}[h]
 \centering
 \includegraphics[scale=0.9]{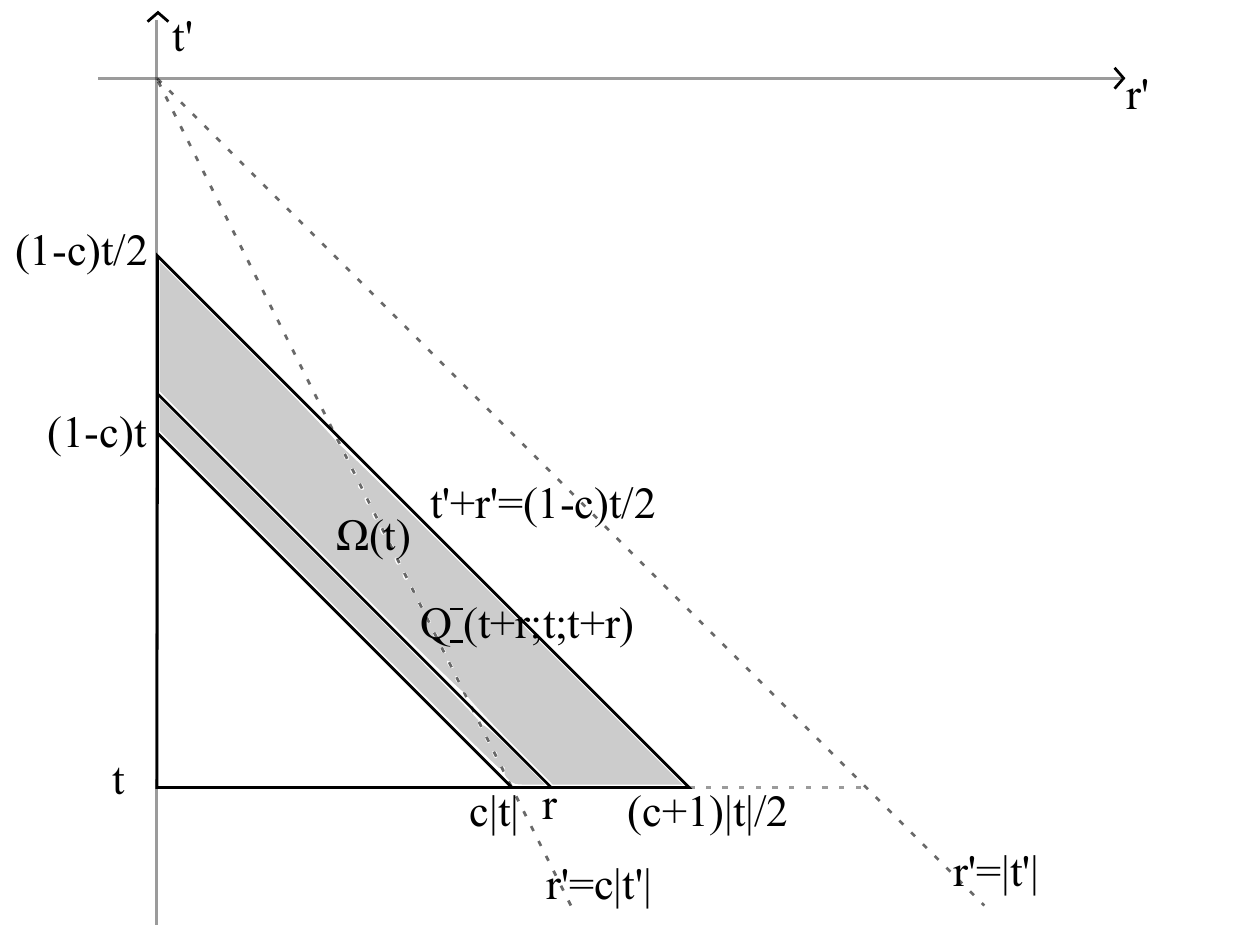}
 \caption{Illustration for proof of Proposition \ref{lower bound of speed}} \label{figure energydis}
\end{figure}

\noindent This shows that the remaining energy is in the sphere shell $\{x: c|t|<|x|<|t|\}$ as $|t|$ is large. The upper bound of $|x|$ can be further improved, so that we can locate the remaining energy in a region further and further away from the light cone $|x|=|t|$ as $t$ goes to $-\infty$. This is the reason why we give this remaining energy another name ``retarded energy''.
\begin{lemma} \label{location of tilde E}
 Given any $\beta < \frac{2(p-2)}{p+1}$, we have the limit
 \begin{align*}
  \lim_{t \rightarrow - \infty} E_{-}(t;|t|-|t|^\beta,+\infty) = \tilde{E}_-.
 \end{align*}
\end{lemma}
\begin{proof} 
Our conclusion is a combination of the following two limits. Here we can choose $s_0=0$ in identity \eqref{limit 2}.
\begin{align}
 \lim_{t \rightarrow - \infty} E_{-}(t;|t|-|t|^\beta,|t|) &= \pi \int_{-\infty}^0 |g_-(s)|^2 ds;& \label{limit 1}\\
 \lim_{t \rightarrow - \infty} E_{-}(t;s_0-t,+\infty) &= \pi \int_{s_0}^\infty |g_-(s)|^2 ds, \quad s_0 \in \Rm.& \label{limit 2}
\end{align}
We start by proving the first one. Let $I = \left(\int_{-\infty}^0 |g_-(s)|^2 ds\right)^{1/2}$. By Corollary \ref{asymptotic behaviour 2} we only need to show 
\[
 \lim_{t \rightarrow -\infty} \left\|(w_r+w_t)(r,t)\right\|_{L^2([|t|-|t|^\beta,|t|])} = \lim_{t \rightarrow -\infty} \left\|(w_r+w_t)(s-t,t)\right\|_{L_s^2([-|t|^\beta,0])} = I.
\] 
Since we have $\displaystyle \lim_{t \rightarrow -\infty} \|g_-\|_{L^2([-|t|^\beta,0])} = I$, it suffices to show 
\[
 \lim_{t\rightarrow -\infty} \left\|(w_r+w_t)(s-t,t) - g_-(s)\right\|_{L_s^2([-|t|^\beta,0])}  = 0. 
\]
This immediately follows the pointwise estimate given in Proposition \ref{limits of wr pm wt}. The same argument as above with the $L^2$ convergence part of Proposition \ref{limits of wr pm wt} instead proves the second limit \eqref{limit 2}.
\end{proof}

\noindent Combining Lemma \ref{lower bound of speed}, Lemma \ref{location of tilde E}, Proposition \ref{monotonicity of energies} and Corollary \ref{asymptotic behaviour 2}, we are able to prove Part (c) of Theorem \ref{main 1}, i.e. we have the following limits for any $c\in (0,1)$ and $0<\beta<\frac{2(p-2)}{p+1}$:
\[
 \lim_{t \rightarrow - \infty} E_{-}(t;c|t|,|t|-|t|^\beta) = \lim_{t \rightarrow -\infty} E(t;c|t|,|t|-|t|^\beta) = E - \tilde{E}_-.
\]

\begin{remark} \label{inside energy limit}
 A combination of identity \eqref{limit 2} and Corollary \ref{asymptotic behaviour 2} gives ($s_0 \in \Rm$)
\[
 \lim_{t \rightarrow -\infty} E_-(t;0,s_0-t) = E-\pi\int_{s_0}^\infty |g_-(s)|^2 ds.
\]
\end{remark}
\begin{remark}
If we apply triangle law on the triangle region $\Omega(s,t_0) = \{(r,t): r>0,t>t_0,r+t<s\}$, make $t_0\rightarrow -\infty$ with Remark \ref{inside energy limit} in mind and finally consider the limit as $s \rightarrow -\infty$, we obtain another expression of the retarded energy 
\[
 E - \tilde{E}_- = \lim_{s \rightarrow -\infty} Q_-^-(s).
\] 
\end{remark}

\section{Scattering with Additional Decay on Initial Data}

In this section we prove Theorem \ref{main 2}, i.e. the solution to (CP1) scatters in both time directions if the initial data satisfy additional decay assumptions. The proof is by a contradiction. If the solution failed to scatter in the negative direction, we would have $\tilde{E}_-<E$. 

\subsection{Additional Contribution by Retarded Energy}

According to Part (c) of Theorem \ref{main 1}, given any $\beta < \frac{2(p-2)}{p+1}$, there exists a negative time $t_1$, so that the inequality
\[
 \int_{|x|<|t|-|t|^\beta} \left(\frac{1}{2}|\nabla u(x,t)|^2 + \frac{1}{2}|u_t(x,t)|^2 + \frac{1}{p+1}|u(x,t)|^{p+1}\right) dx > \frac{E-\tilde{E}_-}{2}.
\]
holds for any time $t<t_1$. If $R$ is a large number $R>|t_1|$ and $t \in (-R-R^\beta,-R)$, we have $R<|t|<R+R^\beta \Rightarrow |t|^\beta>R^\beta$. Thus 
$|t|-|t|^\beta < (R+R^\beta) - R^\beta = R$. This means
\[
 \int_{|x|<R} \left(\frac{1}{2}|\nabla u(x,t)|^2 + \frac{1}{2}|u_t(x,t)|^2 + \frac{1}{p+1}|u(x,t)|^{p+1}\right) dx > \frac{E-\tilde{E}_-}{2}.
\]
As a result we have for sufficiently large $R$:
\begin{equation} \label{additional contribution}
 \int_{-R-R^\beta}^{-R} \int_{|x|<R} \left(\frac{1}{2}|\nabla u(x,t)|^2 + \frac{1}{2}|u_t(x,t)|^2 + \frac{1}{p+1}|u(x,t)|^{p+1}\right) dx dt > \frac{E-\tilde{E}_-}{2}\cdot R^\beta.
\end{equation}
This gives a lower bound of the left hand side of the inequality in Proposition \ref{energy distribution by morawetz}.

\subsection{Upper Bound on Energy Leak}

Now let us give an upper bound on the amount of energy escaping the ball $B(0,R) = \{x \in \Rm^3 :|x|<R\}$ for time $t \in [-R,R]$ under our decay assumption on the energy. In fact we have 
\begin{proposition}
Let $u$ be a solution to (CP1) with a finite energy and satisfy
\[
 I = \int_{\Rm^3} |x|^{\kappa} \left(\frac{1}{2}|\nabla u_0|^2 + \frac{1}{2}|u_1|^2 + \frac{1}{p+1}|u_0|^{p+1}\right) dx < \infty.
\]
Then we have the function
\[
 I(t) = \int_{|x|>|t|} (|x|-|t|)^{\kappa} \left(\frac{1}{2}|\nabla u|^2 + \frac{1}{2}|u_t|^2 + \frac{1}{p+1}|w|^{p+1}\right) dx  \leq I, \quad t\in \Rm.
\]
\end{proposition}
\begin{proof}
 Since the wave equation is time-invertible, it suffices to prove this inequality for $t>0$. A basic calculation then shows $I'(t)\leq 0$ for $t>0$. Please see \cite{pushmorawetz} for more details. 
\end{proof}

\paragraph{Escaping Energy} Given any $t \in (-R,R)$, we have 
\begin{align*}
 & \int_{|x|>R}\left(\frac{1}{2}|\nabla u|^2+\frac{1}{2}|u_t|^2+\frac{1}{p+1}|u|^{p+1}\right) dx\\
 & \qquad \leq (R-|t|)^{-\kappa} \int_{|x|>R} (|x|-|t|)^{\kappa} \left(\frac{1}{2}|\nabla u|^2 + \frac{1}{2}|u_t|^2 + \frac{1}{p+1}|w|^{p+1}\right) dx\\
 & \qquad \leq (R-|t|)^{-\kappa} I(t) \leq (R-|t|)^{-\kappa} I.
\end{align*}
\noindent We integrate $t$ from $-R$ to $R$ and obtain
\begin{equation} \label{final upper bound}
 \int_{-R}^R \int_{|x|>R} \left(\frac{1}{2}|\nabla u|^2+\frac{1}{2}|u_t|^2+\frac{1}{p+1}|u|^{p+1}\right) dx dt \leq \frac{2}{1-\kappa} R^{1-\kappa} I.
\end{equation}
This is an upper bound of the right hand side of the inequality in Proposition \ref{energy distribution by morawetz}. 

\subsection{Completion of the Proof}
Plugging both the lower bound \eqref{additional contribution} and the upper bound \eqref{final upper bound} in the inequality given by Proposition \ref{energy distribution by morawetz}, we obtain the following inequality for any given $\beta<\frac{2(p-2)}{p+1}$ and sufficiently large $R>R(u,\beta)$.
\[
   \frac{E-\tilde{E}_-}{2}\cdot R^\beta \leq \frac{2}{1-\kappa} R^{1-\kappa} I.
\]
If $\kappa>\kappa_0(p) = 1-\frac{2(p-2)}{p+1}$, we can always choose $\beta<\frac{2(p-2)}{p+1}$ so that $\beta>1-\kappa$. This immediately gives a contradiction for sufficiently large $R$'s thus finishes the proof.

\section{Appendix}
In this final section we prove the measure $\mu$ in the energy flux formula satisfies $d\mu = |\xi(t)|^2 dt$. Here $\xi(t)$ satisfies $\pi \|\xi(t)\|_{L^2(\Rm)}^2 \leq E$. Therefore we can substitution $\pi \int_{t_1}^{t_2} 1d\mu(t)$ by $\pi \int_{t_1}^{t_2} |\xi(t)|^2 dt$ in the energy flux formula. It suffices to show the function $P(t) = \mu((\infty,t])$ is an absolutely continuous function. The idea is to find a relation between the measure $\mu$ and the function $g_+(\tau) \in L^2$ introduced in Proposition \ref{limits of wr pm wt}. In fact $\mu$ and $g_+(\tau)$ are limits of $w_r(r,r+\tau)-w_t(r,r+\tau)$ as $r\rightarrow 0^+$ and $r\rightarrow +\infty$, respectively. We start by 
\begin{lemma} \label{upper bound L2}
 Given any interval $I = (\tau_1,\tau_2)$, we always have
 \[
   \mu(I) \leq 2 \int_{\tau_1}^{\tau_2} |g_+(\tau)|^2 d \tau + 2 \int_{\tau_1}^{\tau_2} \left(\int_\tau^\infty \frac{|w(t-\tau,t)|^p}{(t-\tau)^{p-1}} dt\right)^2 d\tau.
 \]
\end{lemma}
\begin{proof}
 First of all, we apply energy flux formula of the outward energy on the region $\{(r,t): 0<r<r', \tau_1<t-r<\tau_2\}$, then let $r' \rightarrow 0^+$ and obtain
\[
  \pi \mu(I) = \pi \lim_{r'\rightarrow 0^+} \int_{\tau_1+r'}^{\tau_2+r'} \left(|w_r(r',t')-w_t(r',t')|^2 - \frac{2}{p+1}\cdot \frac{|w(r',t')|^{p+1}}{{r'}^{p-1}}\right) dt'.
\]
We recall the estimate \eqref{decay of w line}, apply change of variable $\tau = t'-r'$
\[
 \mu(I) = \lim_{r'\rightarrow 0^+} \int_{\tau_1+r'}^{\tau_2+r'} |w_r(r',t')-w_t(r',t')|^2 dt' = \lim_{r'\rightarrow 0^+} \int_{\tau_1}^{\tau_2} |w_r(r',\tau+r')-w_t(r',\tau+r')|^2 d\tau.
\]
By Proposition \ref{var of w partial} and Proposition \ref{limits of wr pm wt} we have the following upper bound
\[
 |w_r(r',\tau+r')-w_t(r',\tau+r')| \leq |g_+(\tau)| + \int_{\tau+r'}^\infty \frac{|w(t-\tau,t)|^p}{(t-\tau)^{p-1}} dt.
\]
We put this upper bound into the expression of $\mu(I)$ above, 
\begin{align*}
 \mu(I) & \leq \limsup_{r'\rightarrow 0^+} \left(2\int_{\tau_1}^{\tau_2} |g_+(\tau)|^2 d\tau + 2\int_{\tau_1}^{\tau_2} \left(\int_{\tau+r'}^\infty \frac{|w(t-\tau,t)|^p}{(t-\tau)^{p-1}} dt\right)^2 d\tau\right)\\
 & \leq 2\int_{\tau_1}^{\tau_2} |g_+(\tau)|^2 d\tau + 2\int_{\tau_1}^{\tau_2} \left(\int_{\tau}^\infty \frac{|w(t-\tau,t)|^p}{(t-\tau)^{p-1}} dt\right)^2 d\tau.
\end{align*}
\end{proof}
\noindent Let us have a look at the upper bound given by Lemma \ref{upper bound L2}. The first term is the integral of an integrable function over the same interval $I$. This will not make any double. Now let us deal with the second term. We first introduce a new notation
\begin{definition}
 Given $\tau\in \Rm$, we define 
 \[
  M(\tau) = \int_{\tau}^\infty \frac{|w(t-\tau,t)|^{p+1}}{(t-\tau)^p} dt.
 \]
\end{definition}
\noindent We claim $M \in L^1(\Rm)$. Because we can apply the change of variable $\tau = t-r$ and obtain
\[
 \int_{-\infty}^{\infty} M(\tau) d\tau = \int_{-\infty}^{\infty} \int_{\tau}^\infty \frac{|w(t-\tau,t)|^{p+1}}{(t-\tau)^p} dt d\tau = \int_{-\infty}^\infty \int_0^\infty \frac{|w(r,t)|^{p+1}}{r^p} dr dt \lesssim_p E. 
\]
\begin{lemma} \label{L2 bound connection}
For all $\tau \in \Rm$ we have
\[
 \int_{\tau}^\infty \frac{|w(t-\tau,t)|^p}{(t-\tau)^{p-1}} dt \lesssim_p Q_+^+(\tau)^{\frac{2}{p+1}}M(\tau)^\frac{p-2}{p+1}.
\]
\end{lemma}
\begin{proof}
We split the integral into two parts 
\[
 \int_{\tau}^\infty \frac{|w(t-\tau,t)|^p}{(t-\tau)^{p-1}} dt = \int_{\tau}^{T} \frac{|w(t-\tau,t)|^p}{(t-\tau)^{p-1}} dt + \int_{T}^\infty \frac{|w(t-\tau,t)|^p}{(t-\tau)^{p-1}} dt
\]
Following the same argument as in the proof of Proposition \ref{pointwise estimate ws and wtau}, we can find an upper bound of the second term above
\[
 \int_{T}^\infty \frac{|w(t-\tau,t)|^p}{(t-\tau)^{p-1}} dt \lesssim_p (T-\tau)^{-\frac{p-2}{p+1}}Q_+^+(\tau)^{\frac{p}{p+1}}.
\]
Next we find an upper bound of the first term using $M(\tau)$
\begin{align*}
 \int_{\tau}^{T} \frac{|w(t-\tau,t)|^p}{(t-\tau)^{p-1}} dt & \leq \left(\int_{\tau}^T \left(\frac{|w(t-\tau,t)|^{p}}{(t-\tau)^{\frac{p^2}{p+1}}}\right)^{\frac{p+1}{p}}dt\right)^{\frac{p}{p+1}} \left(\int_\tau^T \left((t-\tau)^{\frac{1}{p+1}}\right)^{p+1}dt \right)^{\frac{1}{p+1}}\\
 & \leq (T-\tau)^\frac{2}{p+1} M(\tau)^{\frac{p}{p+1}}.
\end{align*}
In summary, we have for any $T>\tau$
\[
  \int_{\tau}^\infty \frac{|w(t-\tau,t)|^p}{(t-\tau)^{p-1}} dt  \lesssim_p (T-\tau)^\frac{2}{p+1} M(\tau)^{\frac{p}{p+1}} +(T-\tau)^{-\frac{p-2}{p+1}}Q_+^+(\tau)^{\frac{p}{p+1}}
\]
Finally we choose $T=\tau+Q_+^+(\tau)/M(\tau)$ and finish the proof.
\end{proof}
\noindent Now we are able to prove
\begin{proposition}
 The function $P(t) = \mu((-\infty,t])$ is an absolutely continuous function. 
\end{proposition}
\begin{proof}
Given any interval $I=(\tau_1,\tau_2)$, we combine Lemma \ref{upper bound L2} and Lemma \ref{L2 bound connection} to obtain
\[
 \mu(I) \lesssim_p \int_{\tau_1}^{\tau_2} |g_+(\tau)|^2 d\tau + E^\frac{4}{p+1} \int_{\tau_1}^{\tau_2} M(\tau)^{\frac{2(p-2)}{p+1}} d\tau.
\]
As a result, if $\displaystyle A = \bigcup_{k=1}^n (a_k,b_k)$ is a union of finite many intervals so that $\displaystyle \sum_{k=1}^n (b_k-a_k)<\delta$, then we have
\begin{align*}
 \sum_{k=1}^n |P(b_k)-P(a_k)| = \mu(A) & \lesssim_p \int_A |g_+(\tau)|^2 d\tau + E^\frac{4}{p+1} \int_A M(\tau)^{\frac{2(p-2)}{p+1}} d\tau\\
 & \lesssim_p \int_A |g_+(\tau)|^2 d\tau + E^\frac{2p}{p+1}\delta^{\frac{5-p}{p+1}}.
\end{align*}
The right hand side above converges to zero, as long as $\delta \rightarrow 0^+$, regardless of the choice of $n$, $a_k$ and $b_k$'s. Therefore $P(t)$ is absolutely continuous by definition. 
\end{proof}
\paragraph{The $L^2$ function $\xi(t)$} Now for any interval $I = (t_1,t_2)$ we are able to write 
\[
 \mu(I) = P(t_2) -P(t_1) = \int_{t_1}^{t_2} P'(t) dt\quad \Rightarrow \quad d\mu(t) = P'(t) dt.
\]
Here $P'(t)$ is a nonnegative, locally integrable function. In addition, the finiteness of $\mu$ implies that $P'(t) \in L^1(\Rm)$. Finally we only need to rewrite $P'(t) = |\xi(t)|^2$, with $\xi(t) \in L^2(\Rm)$ and finish the proof.

\end{document}